\documentclass[a4paper,10pt,reqno]{amsart}

\usepackage{amsmath,graphics}
\usepackage{amssymb}
\usepackage{amsfonts}
\usepackage{mathrsfs}
\usepackage{latexsym}
\usepackage{eucal}
\usepackage[dvips]{graphicx}
\usepackage{relsize}
\usepackage{parskip}

\theoremstyle{plain}
\newtheorem{thm}{Theorem}[section]
\newtheorem{propo}[thm]{Proposition}
\newtheorem{lem}[thm]{Lemma}
\newtheorem{cor}[thm]{Corollary}
\newtheorem{conj}[thm]{Conjecture}

\theoremstyle{definition}

\renewcommand{\Re}{{\rm Re}}
\renewcommand{\Im}{{\rm Im}}
\newcommand{\R}{\mathbb{R}}
\newcommand{\C}{\mathbb{C}}

\newcommand{\N}{\mathbb{N}}

\newcommand{\Prob}{{\mathbb P}}
\newcommand{\Exp}{{\mathbb E}}

\newcommand{\lt}{{\mathscr L}}

\newcommand{\D}{{\mathbb D}}

\newcommand{\HN}{\mathcal{H}_N}

\begin{document}
\bibliographystyle{plain}
\title[Twisted Transfer operators]{Randomly twisted transfer operators and singular values statistics.}

\begin{abstract}
In this paper, we investigate the singular values of a natural family of transfer operators twisted by large random permutation matrices. In the large $N$ limit, we
obtain a Weyl law for its singular values, valid asymptotically almost surely with rapid decay. We also extend the so-called polynomial method to an infinite dimensional setting which implies a "smooth" probabilistic Weyl law for singular values.
\end{abstract}

\author[F.~Naud]{Fr\'ed\'eric Naud}
\address{%
Fr\'ed\'eric Naud\\
Sorbonne universit\'e,\ Institut Math\'ematique de Jussieu Paris-Rive Gauche\\
4, place Jussieu\\
Boite Courrier 247\\
75252 Paris Cedex 05\\
France.
}
\email{frederic.naud@imj-prg.fr}

\subjclass[2000]{37C30, 37D20}

\keywords{Ruelle eigenvalues, Analytic expanding maps, 
Transfer operators, Random matrices}

\maketitle

\tableofcontents

\section{Setting and introduction}
Transfer operators are ubiquitous objects in the study of hyperbolic dynamical systems and their asymptotic statistics. In the early works of the pioneers, they arised through Markov codings and where used as a basic tool to construct invariant measures and establish decay of correlations, see the classics \cite{Bowen_Book, Baladi1,PP}. For analytic dynamical systems, general transfer operators can be described as weighted sums of composition operators acting on suitable holomorphic function spaces. They were pioneered by David Ruelle in \cite{Ruelle}. More recently, we mention the fairly general papers of Bandtlow-Jenkinson \cite{BJ1,BJ2}, see also \cite{BN}, for sharp estimates on the eigenvalues. 

Let us be more specific on the models we want to study here.
Let $\Omega_0$ be a bounded, simply connected, bounded open set in $\C$. We assume that the boundary $\partial \Omega_0$ is at least $C^2$ regular.
Let $\gamma_1,\ldots,\gamma_d:\Omega_0\rightarrow \Omega_0$ be holomorphic mappings, and $G:\Omega_0\rightarrow \C$ be a weight function. We assume that the following holds.
\begin{enumerate}
 \item Each map $\gamma_j$ extends continuously to $\overline{\Omega_0}$.
 \item For all $j=1,\ldots,d$, we have $\overline{\gamma_j(\Omega_0)}\subset \Omega_0$.
 \item The weight $G$ extends continuously to $\overline{\Omega_0}$ and is holomorphic on $\Omega_0$.
\end{enumerate}
In addition, we pick $U_1,\ldots,U_d\in U_N(\C)$ to be {\it permutation matrices}, chosen uniformly and independently in the symetric group $\mathcal{S}_N$. The twisted transfer operator $\lt_N$ acts on holomorphic, vector valued functions $F:\Omega_0\rightarrow \C^N$ by
$$\lt_N(F)(z):=\sum_{j=1}^d e^{G(\gamma_j z)}U_jF(\gamma_j z).$$
These type of operators appear naturally in the geometric context of covers of hyperbolic surfaces, see for example in \cite{MN1, MCN1,PS1}, where their spectral properties are closely related to Laplace spectrum.
We will let $\lt_N$ act on the Bergman space $$\mathcal{H}_N:=H^2(\Omega_0,\C^N),$$ defined by
$$\mathcal{H}_N:=\left\{ F:\Omega_0\rightarrow \C^N\ :\ F\ \mathrm{is\ holomorphic\ and}\ \int_{\Omega_0}\Vert F(z)\Vert^2dm(z)<\infty\right \},$$
where $m$ stands for Lebesgue measure on $\C$ and $\Vert .\Vert$ is the usual Hermitian norm on $\C^N$. The operator $\lt_N$ is then compact and trace class, as composition operators by strict holomorphic contractions have this property, see \cite{BJ1}. The spectrum of $\lt_N$ is therefore discrete outside $\{0\}$ and we denote by $(\lambda_j(\lt_N))_{j\in \N}$ this sequence of eigenvalues with 
$\lim_{j\rightarrow \infty} \lambda_j(\lt_N)=0$. 

\noindent {\bf A typical example}.
Let $\Omega_0$ be the open disc given by 
$$\Omega_0=D\left(1,\frac{3}{2}\right)=\left \{z \in \C\ :\ \vert z-1\vert<\frac{3}{2}\right\}.$$
For all $j=1,\ldots,d\leq \infty$, set $\gamma_j(z)=\frac{1}{j+z}$. Then it is easy to see that each M\"obius transformation $\gamma_j$ satisfies
$$\gamma_j(\Omega_0)\subset D(1,1).$$
 The associated transfer operator with weights $e^{G_j(z)}=\frac{1}{(j+z)^{2\sigma}}$ is often called the Mayer-Perron-Frobenius operator, see \cite{Mayer}, and acts on scalar functions by the popular formula
 $$\lt_{\sigma}(f)(x):=\sum_{j=1}^d \frac{1}{(j+z)^{2\sigma}}f\circ \gamma_j(x).$$
 When $d$ is finite, which is the case we only consider in this paper, it is often associated to the dynamics of continued fractions with bounded entries and the "open" Gauss map. The limit set generated by $\gamma_1,\ldots,\gamma_d$ is a Cantor subset of $[0,1]$ with positive Hausdorff dimension. These transfer operators are standard models of "open chaotic systems" and are useful toy-models in quantum chaos, see for example the paper of Arnoldi-Faure-Weich \cite{AFW}.

In this paper we will focus on the behaviour of spectra and singular values of general operators $\lt_N$ in the large $N$ regime, when permutation matrices $U_1,\ldots,U_d$ are chosen uniformly independently, $d$ being fixed. For all probabilistic statements, $\Prob_N$ will denote the probability of an event while $\Exp_N$ will denote the expectation. An event $\mathcal{E}_N$ is said to occur asymptotically almost surely (a.a.s.) if we have $\lim_{N\rightarrow +\infty} \Prob_N(\mathcal{E}_N)=1$.

A first basic result that we will prove is a deterministic upper bound on the number of eigenvalues as $N$ goes to infinity.
\begin{thm}
\label{Eigest}
 Using the above notations, there exists a constant $C_0>0$ such that for all $r>0$,
 $$\#\{ j\in \N :\ \vert \lambda_j(\lt_N)\vert \geq r^{-1}\}\leq C_0 r N.$$
\end{thm}
This upper bound follows straightforwardly from an estimate of the trace norm of $\lt_N$ and Weyl inequalities, see $\S 2$. A natural and difficult problem is the following conjecture.

\begin{conj} Assume that $d\geq 2$. There exists an absolutely continuous measure $\nu$, compactly supported in $\C$, locally finite in $\C^*$, such that for all $\varphi \in C_0^\infty (\C^*)$, we have
 $$\lim_{N\rightarrow \infty} \frac{1}{N} \Exp_N\left (\sum_{\lambda\in \mathrm{Sp}(\lt_N)} \varphi(\lambda)\right)=\int_{\C} \varphi d\nu.$$
\end{conj}

In the case when $d=1$ and $G=0$, this is not true  because spectra of single composition operators $T_\gamma$ can be explicitely computed and involve powers of the derivative at the unique fixed point of $\gamma$ in $\Omega_0$. The spectrum of $\lt_N$
is therefore the spectrum of the tensor product $T_\gamma\otimes U$ where $U$ is a random permutation matrix, and the limit distribution is not absolutely continous. We will assume that $d\geq 2$ in the whole paper.

Should this conjecture be proved for $d\geq 2$, it could be considered as an infinite dimensional variant of the "single ring theorem", see \cite{GKZ1,BD1}. To understand why this is a hard problem (due to the non-self adjointness), one has to be aware that for the finite dimensional problem of sums of $d$ i.i.d random permutation matrices, the limit law is currently still not known when $d$ is fixed, see in \cite{BCZ1} for more insights. The problem is likely to be slightly easier if $U_1,\ldots,U_d$ are chosen as independent, Haar distributed, unitary matrices, see \cite{BD1}, but the techniques used in \cite{BD1} should be replaced by a softer approach to allow an extension to infinite dimensional operators.

A fist step toward the analysis of non-self adjoint spectra is to study {\it singular values} instead. We recall that if $T:\mathcal{H}\rightarrow \mathcal{H}$ is a compact operator acting on a separable Hilbert space $\mathcal{H}$, the {\it singular values sequence}
$$\mu_0(T)=\Vert T \Vert\geq \mu_1(T)\geq \ldots \geq \mu_j(T)\geq \ldots > 0$$
is by definition the eigenvalue sequence of the positive compact self-adjoint operator $\sqrt{T^* T}$. We prove the following fact.

\begin{thm}
\label{WS} 
Set for all $r>0$, $$\mathcal{N}(r):=\# \left \{ j\in \N\ :\ \mu_j(\lt_N)\geq \frac{1}{r}\right \}.$$
Using the above notations, there exist  constants $C_2>0$ and $r_0>0$ with the following property. For all $k\in \N$, one can find $A_k>0$ such that as $N\rightarrow +\infty$, we have
$$ \Prob_N \left (C_2^{-1}N\leq  \mathcal{N}(r_0)\leq C_2N \right)\geq 1-\frac{A_k}{N^k}.$$
\end{thm}
In other words, with rapid convergence as $N\rightarrow +\infty$, we have an a.a.s Weyl law for the singular values of $\lt_N$. 
The constant $r_0$ can be made effective, for example if $\Omega_0$ is a disc. The proof of Theorem \ref{WS} actually follows readily from the following moments result, which is of independent interest.

\begin{propo}
\label{Moments} Set $\Vert \lt_N\Vert_2^2=\sum_{j=0}^\infty \mu_j^2(\lt_N)$.
\begin{enumerate}
 \item There exists an explicit $L_1=L_1(G,\gamma_1,\ldots,\gamma_d)>0$ such that we have
 $$\lim_{N\rightarrow +\infty} \frac{\Exp_N(\Vert \lt_N\Vert_2^2)}{N}=L_1. $$
 \item Consider the moments 
 $$\mathbb{V}_N^{(k)}:=\Exp_N\left ( \left \vert \Vert \lt_N\Vert_2^2 - \Exp_N(\Vert \lt_N\Vert_2^2)\right \vert^{2k}\right),$$
 then for all $k\geq 1$, $L_{2k}:=\lim_{N\rightarrow \infty} \mathbb{V}_N^{(k)}$ exists and is finite.
 \item Moreover, set for all $i,j\in \{1,\ldots,d\}$, 
 $$H_{i,j}:=\int_{\Omega_0} e^{G\circ \gamma_{i}(z)}e^{\overline{G\circ \gamma_{j}(z)}}B_{\Omega_0}(\gamma_{i}z,\gamma_{j}z)dm(z),$$
 where $B_{\Omega_0}(z,w)$ denotes the Bergman kernel of the space $H^2(\Omega_0,\C)$.
 Consider the random variable
 $$X_N:= \Vert \lt_N\Vert_2^2 - \Exp_N(\Vert \lt_N\Vert_2^2).$$
 Then as $N\rightarrow +\infty$, we have convergence in distribution
 $$X_N\overset{\mathrm{dist}}{\longrightarrow} 2\sum_{1\leq i<j\leq d} (Z_{i,j}-1)\Re(H_{i,j}),$$
 where for all $i<j$, $Z_{i,j}$ are independent Poisson variables with parameter $1$. 
\end{enumerate}
\end{propo}

The fact that all the moments have a limit is by itself non-trivial.
Let us show how to derive Theorem \ref{WS} from Proposition \ref{Moments}. First observe that the upper bound is actually valid deterministically, see Proposition \ref{tracest1}, so we only need to worry about the lower bound. Let $t>0$ to be adjusted later on. By Markov inequality we have for all $N$ large and a constant $A_k>0$, 
$$\Prob_N\left (  \left \vert \Vert \lt_N\Vert_2^2 - \Exp_N(\Vert \lt_N\Vert_2^2)\right \vert \geq t \right)$$ 
$$=\Prob_N\left (  \left \vert \Vert \lt_N\Vert_2^2 - \Exp_N(\Vert \lt_N\Vert_2^2)\right \vert^{2k} \geq t^{2k} \right)\leq \frac{1}{t^{2k}} \Exp_N\left ( \left \vert \Vert \lt_N\Vert_2^2 - \Exp_N(\Vert \lt_N\Vert_2^2)\right \vert^{2k}\right)$$
$$\leq  \frac{A_k}{t^{2k}}.$$
Taking $t=\sqrt{N}$ and dividing by $N$ we get
$$\Prob_N\left (  \frac{\left \vert \Vert \lt_N\Vert_2^2 - \Exp_N(\Vert \lt_N\Vert_2^2)\right \vert}{N} \geq \frac{1}{\sqrt{N}} \right) 
\leq \frac{A_k}{N^{k}}.$$
Therefore using $(1)$ we get for all $N$ large
$$\Prob_N\left ( \frac{\Vert \lt_N\Vert_2^2}{N}\geq \frac{L}{2}- \frac{1}{\sqrt{N}}\right)\geq 1-  \frac{A_k}{N^{k}},$$
which for all $N$ large enough gives
$$\frac{\Vert \lt_N\Vert_2^2}{N}\geq \frac{L}{3},$$
with probability greater than $1-  \frac{A_k}{N^{k}} $. Let us consider the counting function
$$\mathcal{N}(r):=\# \left \{ j\in \N\ :\ \mu_j(\lt_N)\geq \frac{1}{r}\right \},$$
we can write
$$\Vert \lt_N\Vert_2^2=\sum_{j=0}^\infty \mu_j(\lt_h)^2=\sum_{\mu_j > 1/r} \mu_j^2+\sum_{\mu_j \leq 1/r} \mu_j^2,$$
which can be bounded as
$$\Vert \lt_N\Vert_2^2\leq M_0^2\mathcal{N}(r)+\int_r^\infty \frac{d\mathcal{N}(\lambda)}{\lambda^2},$$
by using the fact that $\Vert \lt_N \Vert_{\mathcal{H}_N}\leq M_0$ uniformly in $N$, see below. Using the deterministic upper bound $\mathcal{N}(r)\leq C_3rN$ from section 2, (\ref{singest1}), we get by summation by parts
$$\Vert \lt_N\Vert_2^2\leq M_0^2\mathcal{N}(r)+3\int_r^\infty \frac{\mathcal{N}(\lambda)d\lambda}{\lambda^3}\leq M_0^2\mathcal{N}(r)
+3C_3\frac{N}{r}.$$
Therefore with probability greater than $1-  \frac{A_k}{N^{k}} $, we have
$$\mathcal{N}(r)\geq M_0^{-2}\left ( \frac{L}{3}-\frac{3C_3}{r}\right)N,$$
and we are done as long as $r>9C_3/L$.

Although Theorem \ref{WS} gives the exact growth rate of the counting function $\mathcal{N}(r)$ as $N\rightarrow +\infty$, it does not provide a precise asymptotic description of the location of singular values. We will prove a complementary result, based on a different set of techniques, which we state below. Let $\mathbb{F}^d$ be the {\it abstract free group on $d$ generators}
 $$\mathbb{F}^d:=\langle a_1,a_2,\ldots,a_d;a_1^{-1},\ldots, a_d^{-1}\rangle.$$
 The left regular representation $\lambda:\mathbb{F}^d\rightarrow U(\ell^2(\mathbb{F}^d))$ is defined for all $f\in \ell^2(\mathbb{F}^d)$ and all $\gamma \in \mathbb{F}^d$ by
 $$\lambda(\gamma)(f):=f\circ \gamma^{-1}.$$
 We define a bounded operator $\mathscr{M}_\infty$, acting on the tensor product $H^2(\Omega_0)\otimes \ell^2(\mathbb{F}^d)$, by the formula
 $$\mathscr{M}_\infty:=\left ( \sum_{i=1}^d e^{G\circ \gamma_i}T_{\gamma_i}\otimes \lambda(a_i)\right)^*\left ( \sum_{j=1}^d e^{G\circ \gamma_j}T_{\gamma_j}\otimes \lambda(a_j)\right),$$
 where $T_{\gamma_i}:H^2(\Omega_0)\rightarrow H^2(\Omega_0)$ is the composition operator associated to the contraction $\gamma_i$. Note that we have also the formula
 $$\mathscr{M}_\infty=\left ( \sum_{j=1}^d \vert e^{G\circ \gamma_j}\vert^2 T_{\gamma_j}^*T_{\gamma_j}\right)\otimes Id +
 \sum_{i\neq j} e^{\overline{G\circ \gamma_i}} e^{G\circ \gamma_j}T_{\gamma_i}^*T_{\gamma_j}\otimes \lambda(a_i^{-1}a_j).$$
 By definition, $\mathscr{M}_\infty$ is a positive self-adjoint operator with non-negative spectrum.
 We denote by $V_N$ the subspace of $\mathcal{H}_N=H^2(\Omega_0,\C^N)$ defined by
 $$V_N:=\left\{ F=(F_1,\ldots,F_N)\in \mathcal{H}_N\ :\ \forall\ z\in \Omega_0,\  \sum_{i=1}^NF_i(z)=0\right\}.$$
 The subspace $V_N$ is exactly the orthogonal to the subspace of functions $F\in \mathcal{H}_N$ whose coordinates are all the same i.e. $F(z)=f(z)\otimes (1,\ldots,1)$. Note that both $V_N$ and
 $V_N^\perp$ are invariant by $\lt_N$ since permutations matrices leave constant vectors invariants. 
 
 \begin{thm}
 \label{poly1} There exists a finite positive Radon measure $\mu_\infty$, whose support satisfies
 $$\mathrm{Sp}(\mathscr{M}_\infty)\setminus \{0\}\subset \mathrm{supp}(\mu_\infty)\subset  \mathrm{Sp}(\mathscr{M}_\infty),$$ 
 such that we have for all $C^\infty_0(\R)$ function $\varphi$ with $\varphi(0)=0$,
 $$\lim_{N\rightarrow +\infty} \frac{\Exp_N \left(\mathrm{Tr}(\varphi(\lt_N^*\lt_N\vert_{V_N}))\right)}{N}=\int_{\R} \frac{\varphi(x)}{x} d\mu_{\infty}(x).$$
 \end{thm}
 Note that since $\lt_N^*\lt_N$ is self-adjoint, $\varphi(\lt_N^*\lt_N)$ is well defined for all smooth function, the hypothesis $\varphi(0)=0$ is here to ensure that $\varphi(\lt_N^*\lt_N)$ remains trace class. In the above statement, it is important to
 observe that $\lt_N^*\lt_N$ is restricted to $V_N$, to avoid the obvious spectra coming from constant invariant vectors.
 Let us comment more on the result: the spectrum of the limit operator $\mathscr{M}_\infty$ and the measure $\mu_\infty$ describe the asymptotic distribution of singular values of $\lt_N$ in the large $N$ limit, and the above result shows that a smooth Weyl asymptotic for singular values holds on average. A straightforward \footnote{See the last section for some details.}
 corollary is the following.
 
 \begin{cor}
 Let $I=[a,b]$, with $0<a<b$, be an interval such that $$I\cap\mathrm{Sp}(\mathscr{M}_\infty)\not =\emptyset.$$ Then there exist $C_0, C_1>0$ such that as $N\rightarrow +\infty$,
 $$C_0 N\leq \Exp_N( \mathcal{N}_{I})\leq C_1 N,$$
 where $\mathcal{N}_{I}:=\#\{ j\in \N\ :\ \mu_j^2(\lt_N)\in I \}$. Moreover if $\mu_\infty(\{a\})=\mu_\infty(\{b\})=0$, then we have as $N\rightarrow +\infty$
 $$\lim_{N}\frac{\Exp_N( \mathcal{N}_{I})}{ N}=\int_I \frac{d\mu_{\infty}(x)}{x}. $$
  
 \end{cor}

 It would be very interesting to know if this result can be strenghtened to an a.a.s statement like the counting result of Theorem \ref{WS}, but valid for all smooth test functions $\varphi$ as above. We don't know in general how to compute the spectrum of $\mathscr{M}_\infty$ and if the measure $\mu_\infty$ is absolutely continuous with respect to Lebesgue. We expect that the spectrum of $\mathscr{M}_\infty$ is always (for $d\geq 2$) the full interval $[0,\Vert \mathscr{M}_\infty \Vert]$ and may have some embedded eigenvalues and singular continuous spectrum. 
 We will produce some examples where it is possible to describe the spectrum of $\mathscr{M}_\infty$ in an asymptotic regime of "large contraction", see the very last section of this paper.
 
 {\bf Organization of the paper.} In $\S 2$ below, we collect some basic facts on Bergman spaces and prove some basic estimates for transfer operators, which are enough to obtain Theorem \ref{Eigest}. In $\S 3$, we compute the asymptotic variance of the Hilbert schmidt norm of $\lt_N$ via some known results by Nica and Puder-Zimhoni on the statistics of fixed points of random permutations. It implies an a.a.s. Weyl law for the singular values with a polynomial rate of convergence of the propability to $1$.
 To go further, we compute all the moments of the fluctuations of the Hilbert-Schmidt norm in $\S 4$, based on the combinatorial work of Nica. The classical moments method allow then to deduce Proposition \ref{Moments}. In $\S 5$, we extend the polynomial method of \cite{CGTVH} (or rather one of its consequence) to an abstract framework which allows us to deal with elements in the group ring of the free group with trace class coefficients. This requires a careful analysis of the tracial state involved and the
 extension of the functional calculus (for self-adjoint elements) from polynomials to smooth functions vanishing at $0$. The main result is Theorem \ref{main1}, which is proved by approximating coefficients by finite rank operators (in a quantitative way) and simultaneously using the convergence of expectations of traces given by \cite{CGTVH}. The fact that we work with trace class operators in exponential classes is very helpful in this process. Theorem \ref{poly1} is then just a corollary. Section $\S 6$ is devoted to an example where it is possible to prove some non trivial information about the spectrum of the limit operator $\mathscr{M}_\infty$ via a perturbative argument. 

{\bf Related questions}. In this paper, we have opted for permutation matrices rather than general Haar distributed unitary twists. 
The exact same results should be provable without major issues, by using the Weingarten calculus and techniques used in proofs of strong asymptotic freeness,
see for example in \cite{Collins1}.

\section{Trace norm estimate and deterministic upper bounds}
Most of the estimates pertaining to composition operators on $H^2(\Omega_0)$ will be achieved via Bergman kernel techniques, and we include an overview of the tool box
below. 
\subsection{On Bergman kernels}
We need to recall key facts about Bergman kernels which we summarize here and can all be derived from the elementary properties as explained in 
\cite{Krantz_Bergman}. Let $\Omega$ be an open bounded subset of $\C$, and $H^2(\Omega)$ the associated Bergman space of holomorphic functions.
The  {\it Bergman kernel} is the unique reproducing kernel $B_\Omega(z,w)$ such that for all $f\in H^2(\Omega)$ we have for all $z \in \Omega$,
$$f(z)=\int_\Omega B_\Omega(z,w)f(w)dm(w). $$
Given any Hilbert basis $(\varphi_n)_{n\in \N}$ of the (separable) Hilbert space $H^2(\Omega)$, we have (convergence is uniform on compact subsets of $\Omega$),
$$B_\Omega(z,w)=\sum_{n=0}^\infty \varphi_n(z)\overline{\varphi_n(w)}.$$
In addition, we have the following properties.
\begin{enumerate}
\item For all $z,w\in \Omega$, we have $\vert B_\Omega(z,w)\vert^2\leq \vert B_\Omega(z,z)\vert \vert B_\Omega(w,w)\vert$.
\item For all $z\in \Omega$, $B_\Omega(z,z)=\sup \{\vert f(z)\vert^2\ :\ \Vert f \Vert_{H^2(\Omega)}\leq 1\}$.
\item If $z\in \Omega_1 \subset \Omega_2$, then $B_{\Omega_2}(z,z)\leq B_{\Omega_1}(z,z)$.
\end{enumerate}
Let $x\in \C$ and $r>0$. We also recall that in the case of a disc, $\Omega=D(x,r)$, a Hilbert basis of $H^2(D(x,r))$ is given by the family of functions
$${\bf e}_\ell(z):=\sqrt{\frac{\ell+1}{\pi}}\left ( \frac{z-x}{r}\right)^\ell,$$
while the Bergman reproducing kernel of $H^2(D(x,r))$ is explicitly given by
$$B_{D(x,r)}(z,w)= \frac{1}{\pi} \frac{1}{r^2\left ( 1-\frac{(z-x)\overline{(w-x)}}{r^2}\right)^2},$$
see for example \cite{Krantz_Bergman}.
From these elementary facts we will mainly use the following practical consequence.
\begin{propo}
\label{bergest1}
Assume that $\Omega \subset B(x,r)$ with $x\in \Omega$. Set 
$$\mathrm{dist}(z,\partial \Omega)=\inf_{w\in \Omega^c}\vert z-w\vert.$$
Then we have for all $z\in \Omega$,
$$ \frac{1}{\pi r^2}\leq B_\Omega(z,z)\leq \frac{1}{\pi \mathrm{dist}(z,\partial \Omega)^2}$$
\end{propo}
\noindent {\it Proof}. We have $D(z, \mathrm{dist}(z,\partial \Omega))\subset \Omega$ therefore setting $\alpha=\mathrm{dist}(z,\partial \Omega)$, we have
$$B_\Omega(z,z)\leq B_{D(z,\alpha)}(z,z)=\frac{1}{\pi\alpha^2}.$$
On the other hand we have
$$B_\Omega(z,z)\geq B_{D(x,r)}(z,z)=\frac{1}{\pi r^2(1-\vert z-x\vert^2/r^2)^2}\geq   \frac{1}{\pi r^2},$$
the proof is done. $\square$

Stronger estimates can be derived by using the conformal mapping theorem.
Indeed, since $\Omega_0$ has a $C^2$-regular boundary, it is a "Dini smooth" domain and there exists a conformal mapping $\psi:\Omega_0\rightarrow \D$,
where $\D$ is the unit disc, such that $\psi$ and $\psi'$ extend continuously to $\partial \Omega_0$. See for example in the book by Pommerenke \cite{PM1}, Theorem 3.5. It allows to define an isometry $\mathscr{J}:H^2(\D)\rightarrow H^2(\Omega_0)$ via the formula:
$$\mathscr{J}(f):=\phi'(f\circ \psi).$$
As an application, the Bergman kernel of $H^2(\Omega_0)$ is then given by
$$B_{\Omega_0}(z,w)=\frac{\psi'(z)\overline{\psi'(w)}}{\pi\left(1-\psi(z)\overline{\psi(w)}\right)^2}.$$
In particular, if $K$ is any compact subset of $\Omega_0$, setting $R=\max_{z\in K}\vert \psi(z)\vert<1$, we have for all $z\in K$ and all $w\in \Omega_0$,
\begin{equation}
\label{bergest2}
\vert B_{\Omega_0}(z,w)\vert\leq \sup_{\Omega_0}\vert \psi' \vert^2 \frac{1}{\pi(1-R)^2}.
\end{equation}

\subsection{A crude operator norm bound}
In this subsection we prove the following crude bound.
\begin{propo}
There exists a constant $M_0>0$ such that for all $N$ we have
$$\Vert \lt_N \Vert_{\mathcal{H}_N}\leq M_0. $$
\end{propo}
\noindent {\it Proof}. Denote by $T_{\gamma_j}:\HN\rightarrow \HN$ the composition operator defined by $$T_{\gamma_j}(F)=F\circ \gamma_j.$$
Since we can write
$$\lt_N=\sum_{j=1}^d e^{G\circ \gamma_j}U_j\otimes T_{\gamma_j},$$
it is enough to bound each
$$\Vert e^{G\circ \gamma_j}U_j\otimes T_{\gamma_j}\Vert_{\HN}. $$
Writing by unitarity of $U_j$
$$\Vert e^{G\circ \gamma_j}U_j\otimes T_{\gamma_j}(F)\Vert_{\HN}^2=\int_{\Omega_0}\Vert e^{G\circ \gamma_j(z)}U_j F\circ \gamma_j\Vert^2dm(z)$$
$$\leq \sup_{\Omega_0} e^{2\vert G\vert} \int_{\Omega_0}\Vert F\circ \gamma_j\Vert^2dm(z).$$
Using the reproducing property of the Bergman kernel we have for all $z\in \Omega_0$,
$$F(z)=\int_{\Omega_0}Id_N\otimes B_{\Omega_0}(z,w)F(w)dm(w), $$
and therefore by (\ref{bergest2}) and Cauchy-Schwarz we have
$$\sup_{z\in \Omega_0}\Vert F(\gamma_j(z))\Vert\leq  \int_{\Omega_0} \vert B_{\Omega_0}(\gamma_j z,w)\vert \Vert F(w)\Vert dm(w)$$
$$ \leq \sup_{\Omega_0}\vert \psi' \vert^2 \frac{\mathrm{Vol}(\Omega_0)}{\pi(1-R_j)^2} \Vert F \Vert_{\HN}$$
with $R_j=\sup_{z\in \Omega_0}\vert \psi(\gamma_j z)\vert$. All these estimates are uniform in $N$ and we are done. $\square$

\subsection{Bounding the trace norm of $\lt_N$}
In this subsection we will prove the following fact. We recall that given a trace class operator $T:\mathcal{H}\rightarrow \mathcal{H}$ on a Hilbert space $\mathcal{H}$, the Schatten norms are defined for $p=1,\ldots,\infty$ by
$$\Vert T \Vert_p:=\left ( \sum_{j=0}^\infty \mu_j(T)^p\right)^{1/p}.$$
The norm for $p=1$ is called the trace norm while for $p=2$ it is nothing but the Hilbert-Schmidt norm.
\begin{propo}
\label{tracest1}
There exists $C_3>0$ such that for all $N$ large we have
$$\Vert \lt_N \Vert_{1}\leq C_3N.$$
\end{propo}
\noindent{\it Proof}. Recall that by definition we have 
$$\Vert \lt_N \Vert_1=\sum_j \mu_j=\mathrm{Tr}(\sqrt{\lt^*_N\lt_N}e_\ell)=\sum_\ell \langle \sqrt{\lt^*_N\lt_N}e_\ell,e_\ell\rangle, $$
where $(e_\ell)_{\ell \in \N}$ is any Hilbert basis of $\mathcal{H}_N$. By Cauchy-Schwarz this is bounded as
$$\Vert \lt_N \Vert_1\leq \sum_\ell \Vert  \sqrt{\lt^*_N\lt_N}e_\ell \Vert=\sum_\ell \Vert  \lt_Ne_\ell \Vert.$$
We will carry this estimate by using a rather "explicit" basis of $H^2(\Omega_0)$ via the Riemann mapping Theorem (since $\Omega_0$ is simply connected).
Recall the $\psi$ is a conformal mapping $\psi:\Omega_0\rightarrow \D$,
where $\D$ is the unit disc, such that $\psi$ and $\psi'$ extend continuously to $\partial \Omega_0$. 
A basis of $H^2(\Omega_0)$ is obtained for $\ell=0,\ldots,\infty$ as 
$${\mathbf g}_\ell(z):=\psi'(z)\sqrt{\frac{\ell+1}{\pi}}\left ( \psi(z)\right)^\ell.$$
A full Hilbert basis of $\mathcal{H}$ is then given for $j=1,\ldots,N$ and $\ell=0,\ldots,\infty$ by
$$\delta_j\otimes {\mathbf g}_\ell, $$
where $\delta_1,\ldots,\delta_N$ is the canonical basis of $\C^N$. We have therefore to estimate
$$\sum_{j=1}^N\sum_{\ell=0}^\infty \Vert \lt_N(\delta_j\otimes {\mathbf g}_\ell)\Vert_{\mathcal{H}_N}.$$
Let us bound
$$\Vert \lt_N(\delta_j\otimes {\mathbf g}_\ell)\Vert_{\mathcal{H}_N}^2=\int_{\Omega_0}\Vert \lt_N(\delta_j\otimes {\mathbf g}_\ell)\Vert^2dm(z)$$
$$\leq d \sup_{z\in \Omega_0} e^{2\vert G(z)\vert} \sum_{k=1}^d \int_{\Omega_0}\Vert U_k(\delta_j)  {\mathbf g}_\ell(\gamma_k z)\Vert^2dm(z),$$
where we have used the standard inequality
$$\Vert Z_1+\ldots+Z_d \Vert^2\leq d \sum_{k=1}^d\Vert Z_k\Vert^2, $$
valid for all $Z_1,\ldots,Z_d\in \C^N$, which follows directly from Cauchy-Schwarz. Using unitarity of each $U_k$, we are left with estimating
$$\int_{\Omega_0}\vert{\mathbf g}_\ell(\gamma_k z)\vert^2dm(z)=\frac{\ell+1}{\pi}\int_{\Omega_0}\vert \psi'(\gamma_kz)(\psi(\gamma_kz))^\ell\vert^2 dm(z),$$
$$\leq \frac{\ell+1}{\pi} \sup_{\Omega_0}\vert \psi'\vert^2 \int_{\Omega_0}\vert \psi(\gamma_kz))\vert^{2\ell} dm(z).$$
By assumption each $\gamma_k(\Omega_0)$ has a {\it compact closure} inside $\Omega_0$, therefore there exists $0<\rho<1$, independent of $N$ such that
$$ \max_k \sup_{z\in \Omega_0}\vert \psi(\gamma_kz))\vert^{2\ell}\leq \rho^{2\ell}.$$
As a conclusion, there exists a large constant $C>0$ such that we have (uniformly in $N$)
$$ \sum_{j=1}^N\sum_{\ell=0}^\infty \Vert \lt_N(\delta_j\otimes {\mathbf g}_\ell)\Vert_{\mathcal{H}_N}\leq C \sum_{j=1}^N \sum_{\ell=0}^\infty \sqrt{\ell+1}\rho^\ell= \widetilde{C}N,$$
and the proof is done. $\square$

Note that in general, $\psi$ is not explicit and therefore $\widetilde{C}$ is not fully explicit, unless $\Omega_0$ is a disc, in which case it is fully computable. 

\bigskip Let us end this section by giving a very short proof of Theorem \ref{Eigest}. By writing in the style of Markov inequality
$$ \#\{ j\in \N :\ \vert \lambda_j(\lt_N)\vert \geq r^{-1}\}=\sum_{\vert \lambda_j\vert \geq 1/r}1
\leq r \sum_{\vert \lambda_j\vert \geq 1/r}\vert \lambda_j\vert \leq r \sum_{j=0}^\infty \vert \lambda_j\vert,$$
one can then use Weyl inequalities, see \cite{BSimon} 1.6, to get 
$$ \#\{ j\in \N :\ \vert \lambda_j(\lt_N)\vert \geq r^{-1}\}\leq r \Vert \lt_N\Vert_1\leq C_3 rN,$$
and the proof is now complete. 

Note that we have also obtained, by the exact same idea, the deterministic upper bound for all $r>0$,
\begin{equation}
\label{singest1}
\mathcal{N}(r):=\# \left \{ j\in \N\ :\ \mu_j(\lt_N)\geq \frac{1}{r}\right \}\leq C_3 rN,
\end{equation}
which will be useful below (and was used at the end of $\S 1$).
\section{Expectation and variance of  $\Vert \lt_N \Vert_2^2$}
The lower bound of $\mathcal{N}(r)$ from Theorem \ref{WS} will follow from the asymptotic analysis of the Hilbert-Schmidt norm (squared) $\Vert \lt_N \Vert^2_2=\sum_{j=0}^\infty \mu_j^2(\lt_N).$
We first need to introduce the necessary probabilistic tools involved in the proofs.
\subsection{Random permutations and fixed points statistics}
 As stated in the introduction, permutations matrices $U_1,\ldots,U_d$ are chosen randomly, according to the uniform measure on $\mathcal{S}_N$, and independently. Alternatively, we can view it as choosing a random homomorphism $\phi_N:\mathbb{F}^d\rightarrow \mathcal{S}_N$, where $\mathbb{F}^d$ is the {\it abstract free group on $d$ generators}
 $$\mathbb{F}^d:=\langle a_1,a_2,\ldots,a_d;a_1^{-1},\ldots, a_d^{-1}\rangle.$$
 Indeed, for each generator $a_j$ we choose $\phi_N(a_j)$ to be a permutation $\sigma_j\in \mathcal{S}_N$ taken uniformly among $N!$ possible choices. Since there are no non-trivial relations in $\mathbb{F}^d$, the homomorphism $\phi_N$ is then uniquely defined. This leads to a natural probability measure on $\mathrm{Hom}(\mathbb{F}^d,\mathcal{S}_N)$, denoted by $\Prob_N$.
 According to this notations, we have that for all $j=1,\ldots,d$, $U_j$ is the permutation matrix associated to $\sigma_j:=\phi_N(a_j)$. We will clarify notations by denoting $U_j$ by $U_{\sigma_j}$ instead.
 Given a word $\gamma=\gamma_1\gamma_2\ldots \gamma_p\in \mathbb{F}^d$ with $\gamma_i\in \{a_1,a_2,\ldots,a_d,a_1^{-1},\ldots, a_d^{-1}\}$, we will denote by $U_{\phi_N(\gamma)}$ the permutation matrix associated to $\phi_N(b)\in \mathcal{S}_N$. Note that we always have the identity
 $$\mathrm{Tr}(U_{\phi_N(\gamma)})=F_N(\gamma), $$
 where $F_N(\gamma)$ is the {\it number of fixed points} of the permutation $\phi_N(\gamma) \in \mathcal{S}_N$. It is classical exercise in combinatorics to show that, when taken uniformly, the expectation of the number of fixed points of a random permutation is exactly $1$. In our situation, for longer words $\gamma \in \mathbb{F}^d$, the distribution of $F_N(\gamma)$ is in general not uniform. 
Our main tool is the following result. We will denote by
 $\mathbb{F}^d_0$ the set of {\it primitive words} in $\mathbb{F}^d$, i.e. words which are not a non-trivial power of another element. The identity word does not belong to $\mathbb{F}^d_0$.
 
 \begin{propo}
\label{Stat1} Under the above notations, the following holds.
\begin{enumerate}
\item For all $\gamma \in \mathbb{F}^d_0$ and $k\geq 1$, we have as $N$ goes to $\infty$,
$$\Exp_N( F_n(\gamma^k))=d(k)+O\left( \frac{1}{N} \right),$$
where $d(k)$ denotes the number of divisors of $k$.
\item For all $\gamma_1,\gamma_2 \in \mathbb{F}^d_0$ such that $\gamma_1\not \in \{\gamma_2,\gamma_2^{-1} \}$, for all $k_1,k_2 \geq 1$, we have as $N\rightarrow \infty$,
$$\Exp_N( F_N(\gamma_1^{k_1}) F_N(\gamma_2^{k_2}))=\Exp_N( F_n(\gamma_1^{k_1}))\Exp_N( F_n(\gamma_2^{k_2}))+O\left( \frac{1}{N} \right).$$
\item For all $k_1,k_2\geq 1$ and $\gamma \in \mathcal{P}$, the limit 
$$\lim_{N\rightarrow +\infty} \Exp_N\left [  \left(F_N(\gamma^{k_1})-\Exp_N(F_N(\gamma^{k_1}))\right)\left(F_N(\gamma^{k_2})-\Exp_N(F_N(\gamma^{k_2}))\right)\right]$$
$$:=\mathcal{V}(k_1,k_2)$$
exists. If $k_1=k_2=1$ then $\mathcal{V}(1,1)=1$. Moreover  we have the formula
$$\mathcal{V}(k_1,k_2)=\sum_{k \vert \gcd(k_1,k_2)} k.$$
\end{enumerate}
\end{propo}
Let us comment on this above Proposition. In $(2)$, the remainder is not uniform with respect to $\gamma_1,\gamma_2$ and $k_1,k_2$. Notice that since $F_n(\gamma^k)=F_n(\gamma^{-k})$, one can also
use $(3)$ if $k_1,k_2$ are not both positive. The major output of the above proposition is that, in the large $N$ limit, the averages and covariances of the random variables $F_n(\gamma^k)$ asymptotically depend only on the powers $k$'s. This statement follows from the work of Nica \cite{Nica1}, see also the more general paper of Puder and Zimhoni \cite{PZ1}. For the exact derivation of the above Proposition from \cite{PZ1}, see \cite{Naud1}.

\subsection{Trace formula and expectation of $\Vert \lt_N \Vert_2^2$}
Let us compute $\Vert \lt_N \Vert_2^2$ deterministically. Let $(\delta_j \otimes {\bf g}_\ell)$, with $j=1,\ldots,N$ and $\ell=0,\ldots,\infty$ be a Hilbert basis of $\HN$, as in the proof of Proposition \ref{tracest1}.
We know that
$$\Vert \lt_N \Vert_2^2=\mathrm{Tr}(\lt_N^*\lt_N)=\sum_{\ell,j} \Vert \lt_N(\delta_j \otimes {\bf g}_\ell) \Vert_{\HN}^2.$$
Writing
$$\Vert \lt_N(\delta_j \otimes {\bf g}_\ell) \Vert_{\HN}^2=\sum_{1\leq k_1,k_2\leq d} \int_{\Omega_0} e^{G\circ \gamma_{k_1}}e^{\overline{G\circ \gamma_{k_2}}}\langle U_{k_2}^*U_{k_1} \delta_j,\delta_j\rangle_{\C^N}
{\bf g}_\ell\circ \gamma_{k_1} \overline{{\bf g}_\ell\circ \gamma_{k_2}}dm,$$
we can then use Fubini and uniform convergence of $\sum_\ell {\bf g}_\ell(z) \overline{{\bf g}_\ell(w)}$ to $B_{\Omega_0}(z,w)$ on compact subsets of $\Omega_0$ to get the trace formula:
\begin{equation}
\label{trace1}
\Vert \lt_N \Vert_2^2=\sum_{1\leq k_1,k_2\leq d} \mathrm{Tr}(U_{k_2}^*U_{k_1})\int_{\Omega_0} e^{G\circ \gamma_{k_1}(z)}e^{\overline{G\circ \gamma_{k_2}(z)}}B_{\Omega_0}(\gamma_{k_1}z,\gamma_{k_2}z)dm(z). 
\end{equation}
We can now prove the first claim of Proposition \ref{Moments}. Observe that according to our notations, we have
$$ \mathrm{Tr}(U_{k_2}^*U_{k_1})=\mathrm{Tr}(U_{\phi_N(a_{k_2}^{-1}a_{k_1})})=F_N(a_{k_2}^{-1}a_{k_1}),$$
and therefore for all $k_1\neq k_2$ we have by Proposition \ref{Stat1} $(1)$,
$$\lim_{N\rightarrow \infty} \frac{\Exp_N(\mathrm{Tr}(U_{k_2}^*U_{k_1}))}{N}=\lim_{N\rightarrow \infty} \frac{\Exp_N(F_N(a_{k_2}^{-1}a_{k_1}))}{N}=0.$$
This leads straightforwardly to the formula
$$\lim_{N\rightarrow \infty} \frac{\Exp_N(\Vert \lt_N \Vert_2^2)}{N}=\sum_{j=1}^d  \int_{\Omega_0}\left\vert e^{G\circ \gamma_{j}(z)}\right \vert^2
B_{\Omega_0}(\gamma_{j}z,\gamma_{j}z)dm(z). $$
Note that the constant
$$L_1:=\sum_{j=1}^d  \int_{\Omega_0}\left\vert e^{G\circ \gamma_{j}(z)}\right \vert^2
B_{\Omega_0}(\gamma_{j}z,\gamma_{j}z)dm(z),$$
is indeed positive, by positivity of the Bergman kernel on the diagonal, see Proposition \ref{bergest1}. We have actually the lower bound
$$L_1\geq d e^{-2\sup_{\Omega_0}\vert G\vert} \frac{\mathrm{Vol}(\Omega_0)}{\pi \mathbf{R_0}^2},$$
where $\mathbf{R_0}$ is such that $\Omega_0\subset D(x_0,\mathbf{R_0})$, by the lower bound from Proposition \ref{bergest1}.

\subsection{Limit Variance of $\Vert \lt_N \Vert_2^2$}
We now turn our attention to the variance 
$$\mathbb{V}_N^{(2)}:=\Exp_N\left ( \left \vert \Vert \lt_N\Vert_2^2 - \Exp_N(\Vert \lt_N\Vert_2^2)\right \vert^{2}\right)$$
$$=\Exp_N\left (  \Vert \lt_N\Vert_2^4 \right)-\left (\Exp_N\left (  \Vert \lt_N\Vert_2^2 \right) \right)^2.$$
We will actually show, using Proposition \ref{Stat1}, that $\lim_{N\rightarrow +\infty} \mathbb{V}_N^{(2)}$ exists and has an explicit formula.
Let us set for all $i,j\in \{1,\ldots,d\}$,
$$H_{i,j}:=\int_{\Omega_0} e^{G\circ \gamma_{i}(z)}e^{\overline{G\circ \gamma_{j}(z)}}B_{\Omega_0}(\gamma_{i}z,\gamma_{j}z)dm(z). $$
Note that we always have $H_{j,i}=\overline{H_{i,j}}$. By the trace formula (\ref{trace1}), we have
$$\mathbb{V}_N^{(2)}=$$
$$\sum_{1\leq i_1,i_2\leq d \atop 1\leq j_1,j_2\leq d} \left \{ \Exp_N(F_N(a_{i_1}^{-1}a_{i_2})F_N(a_{j_1}^{-1}a_{j_2}))-  
\Exp_N(F_N(a_{i_1}^{-1}a_{i_2}) )\Exp_N(F_N(a_{j_1}^{-1}a_{j_2}))  \right \}H_{i_1,i_2} \overline{H_{j_1,j_2}}.$$
We observe first that if $i_1=i_2$ then we get
$$\Exp_N(F_N(a_{i_1}^{-1}a_{i_2})F_N(a_{j_1}^{-1}a_{j_2}))-  
\Exp_N(F_N(a_{i_1}^{-1}a_{i_2}) )\Exp_N(F_N(a_{j_1}^{-1}a_{j_2}))$$
$$=\Exp_N(NF_N(a_{j_1}^{-1}a_{j_2}))-N\Exp_N(F_N(a_{j_1}^{-1}a_{j_2}))=0.$$
Same outcome if $j_1=j_2$. We are thus left with the sum
$$ \mathbb{V}_N^{(2)}=$$
$$\sum_{ i_1\neq i_2 \atop j_1\neq j_2} \left \{ \Exp_N(F_N(a_{i_1}^{-1}a_{i_2})F_N(a_{j_1}^{-1}a_{j_2}))-  
\Exp_N(F_N(a_{i_1}^{-1}a_{i_2}) )\Exp_N(F_N(a_{j_1}^{-1}a_{j_2}))  \right \}H_{i_1,i_2} \overline{H_{j_1,j_2}}.$$
Assuming that $(i_1,i_2)\not \in \{(j_1,j_2),(j_2,j_1) \}$ then since both words $a_{i_1}^{-1}a_{i_2}$ and $a_{j_1}^{-1}a_{j_2}$ are reduced in $\mathbb{F}^d$, 
we know that $a_{i_1}^{-1}a_{i_2}\not \in \{ a_{j_1}^{-1}a_{j_2}, (a_{j_1}^{-1}a_{j_2})^{-1} \}$. We can therefore apply Proposition \ref{Stat1} $(2)$ to get
$$\lim_{N\rightarrow +\infty } \left (\Exp_N(F_N(a_{i_1}^{-1}a_{i_2})F_N(a_{j_1}^{-1}a_{j_2}))-  
\Exp_N(F_N(a_{i_1}^{-1}a_{i_2}) )\Exp_N(F_N(a_{j_1}^{-1}a_{j_2}))\right) =0.$$
On the other hand, if $(i_1,i_2) \in \{(j_1,j_2),(j_2,j_1) \}$, noticing that the words $a_{i_1}^{-1}a_{i_2}$ are primitive (not proper powers),  and using the fact that
$$F_N(a_{i_1}^{-1}a_{i_2})=F_N((a_{i_1}^{-1}a_{i_2})^{-1}),$$
we can apply 
Proposition \ref{Stat1} $(3)$ which yields in both cases $(i_1,i_2)=(j_1,j_2)$, $(i_1,i_2)=(j_2,j_1)$
$$\lim_{N\rightarrow +\infty } \Exp_N\left ((F_N(a_{i_1}^{-1}a_{i_2}))^2 \right)-\left(\Exp_N(F_N(a_{i_1}^{-1}a_{i_2}) )\right)^2=\mathcal{V}(1,1)=1.$$
Therefore we have obtained
$$L_2:=\lim_{N\rightarrow +\infty}\mathbb{V}_N^{(2)}=\sum_{i_1\neq i_2} H_{i_1,i_2}(\overline{H_{i_1,i_2}}+ \overline{H_{i_2,i_1}})$$
$$=\sum_{i_1\neq i_2} H_{i_1,i_2}(\overline{H_{i_1,i_2}}+ H_{i_1,i_2})=2\sum_{i_1\neq i_2} H_{i_1,i_2}\Re(H_{i_1,i_2})$$
$$=4\sum_{i_1<i_2}  (\Re(H_{i_1,i_2}))^2.$$

One may wonder at this point if $L_2$ is in general positive. Going back to the example of the introduction, we have in this case
$$H_{j,j+1}=\frac{4}{9\pi} \int_{D(1,3/2)} \frac{dm(z)}{\left ( 1-\frac{4}{9}( \gamma_j(z)-1)( \overline{\gamma_{j+1}(z)}-1)\right)^2}.$$
We have the bound 
$$\sup_{z\in \Omega_0}\vert \gamma_j(z)\vert \leq \frac{1}{j-1/2},$$ 
hence uniform convergence on $\Omega_0=D(1,3/2)$ of $\gamma_j(z)$ to $0$ as $j\rightarrow +\infty$. This allows us to obtain
$$\lim_{j\rightarrow +\infty} H_{j,j+1}=\frac{4}{9\pi} \int_{D(1,3/2)} \frac{dm(z)}{\left ( 1-\frac{4}{9}\right)^2}=\frac{81}{25}.$$
Therefore if $d$ is taken large enough, for the example stated in the introduction, we have $L_2>0$.

\section{Higher moments and Poisson distributions}
Let us set in the following
$$X_N:= \Vert \lt_N\Vert_2^2 - \Exp_N(\Vert \lt_N\Vert_2^2).$$
We recall that we have, using the above notations
$$X_N=2\sum_{1\leq i<j\leq d}\left( F_N(a_i^{-1}a_j)-\Exp_N(F_N(a_i^{-1}a_j))\right) \Re(H_{i,j}),$$
and using Proposition \ref{Stat1} (each word $a_i^{-1}a_j$ is primitive), we know that for all $i<j$,
$$\Exp_N(F_N(a_i^{-1}a_j))=1+O\left( \frac{1}{N}\right),$$
therefore we have as $N\rightarrow +\infty$,
$$X_N=2Y_N-2\sum_{1\leq i<j\leq d}\Re(H_{i,j})+O\left( \frac{1}{N}\right),$$
where we have set
$$Y_N:= \sum_{1\leq i<j\leq d}F_N(a_i^{-1}a_j)\Re(H_{i,j}).$$
We will {\it compute below the limit of each moments of $Y_N$}, which imply convergence of all moments of $X_N$, and then use a classical result to obtain convergence in distribution
to a linear combination of independent Poisson variables, which is part $(3)$ of Proposition \ref{Moments}. We first recall basic facts about linear combinations of independent Poisson variables.
\subsection{Poisson variables}
A Poisson variable $Z$ with parameter $\lambda$ obeys the probability law for $k=0,\ldots,\infty$
$$\Prob(Z=k)=\frac{e^{-\lambda}\lambda^k}{k!}.$$
The (non-centered) moments are given by
$$\Exp\left( Z^{k}\right)=\sum_{\ell=0}^k S(k,\ell)\lambda^\ell,$$
where $S(k,\ell)$ are so-called Stirling numbers, which is the number of ways to partition a set of $k$ objects into $\ell$ non-empty subsets.

The case $\lambda=1$, which is especially relevant for us, gives
$$\Exp\left( Z^{k}\right)=B_k,$$
where $B_k$ is the Bell number (which counts the total number of partitions of set of $k$ objects) defined by the recursion formula
$$B_{k+1}=\sum_{\ell=0}^k \binom{k}{\ell}B_\ell, $$
and do satisfy the classical Dobinski formula:
$$B_k=e^{-1}\sum_{\ell=0}^\infty \frac{\ell^k}{\ell!},$$
which is just another way of writing 
$$\Exp(Z^k)=\sum_{\ell=0}^\infty \Prob(Z=\ell) \ell^k.$$
It is also important to recall that the moment generating function is known
$$\Exp(e^{tZ})=\sum_{\ell=0}^\infty \frac{t^\ell \Exp(Z^\ell)}{\ell!}=e^{e^t-1},$$
and is in particular finite for all $t\in \R$.
In this section, we are interested in identifying the moments of a linear combination
$$Y=\sum_{p=1}^m\alpha_p Z_p,$$
where  $Z_1,\ldots,Z_p$ are independent Poisson variables with parameter $\lambda=1$ and $\alpha_1,\ldots,\alpha_m \in \R$.
By the multinomial formula and independence of $Z_1,\ldots,Z_m$, we clearly have
$$\Exp(Y^k)=\sum_{\ell_1+\ell_2+\ldots+\ell_m=k} \frac{k!}{\ell_1!\ldots \ell_m!} \alpha_1^{\ell_1}\ldots \alpha_m^{\ell_m} \Exp\left (Z_1^{\ell_1}\ldots Z_m^{\ell_m} \right)$$
\begin{equation}
\label{momentform1}
=k! \sum_{\ell_1+\ell_2+\ldots+\ell_m=k} \frac{ \alpha_1^{\ell_1}}{\ell_1!}\ldots \frac{ \alpha_m^{\ell_m}}{\ell_m!}B_{\ell_1}\ldots B_{\ell_m}.
\end{equation}
Note that we have the bound for each $k\geq 0$ and $t\geq 0$,
$$\frac{\vert \Exp(Y^k)\vert t^k}{k!}\leq \sum_{\ell_1+\ell_2+\ldots+\ell_m=k} \frac{ (\vert\alpha_1\vert t)^{\ell_1}}{\ell_1!}\ldots \frac{ (\vert \alpha_m\vert t)^{\ell_m}}{\ell_m!}B_{\ell_1}\ldots B_{\ell_m},$$
and thus
\begin{equation}
\label{uniqueness1}
\sum_{k=0}^\infty  \frac{\vert \Exp(Y^k)\vert t^k}{k!} \leq \prod_{j=1}^m \left ( \sum_{\ell=0}^\infty \frac{(\vert \alpha_j\vert t )^\ell}{\ell!}B_\ell \right)
=e^{-m}\exp\left (\sum_{j=1}^m e^{t\vert \alpha_j\vert}\right).
\end{equation}
We therefore have convergence of 
$$\Exp( e^{tY})=\sum_{k=0}^\infty  \frac{ \Exp(Y^k) t^k}{k!} $$
for all $t\in \R$.
\subsection{The method of moments} 
We will apply below the following classical facts on moments of random variables. A basic reference is for example the book of Billingsley \cite{Bill1}.
\begin{thm}
\label{Momentmethod}
\begin{itemize}
\item A Borel probability measure $\mu$ on $\R$ is said to be determined by its moments if for all $k\in \N$,
$$a_k:=\int_\R x^k d\mu(x)$$
exists and $\mu$ is the only probability measure with this sequence of moments $(a_k)_{k\in \N}$.
\item A sufficient condition for a probability measure $\mu$ to be determined by its moments is the absolute convergence of the series
$$\sum_{k=0}^\infty a_k \frac{t^k}{k!}$$
for some $t>0$.
\item Let $(W_N)$ be a sequence of real random variables, defined on  probability spaces $(\Omega_N,\Prob_N)$ and $W$ a real random variable, defined on $(\Omega, \Prob)$, whose distribution is determined by its moments.
If we have for all $k\in \N$,
$$\lim_{N\rightarrow +\infty }\Exp_N(W_N^k)=\Exp(W^k), $$
then $W_N$ converges as $N\rightarrow \infty$ to $W$ in distribution.
\end{itemize}
\end{thm}
Convergence in distribution means that for all $x\in \R$ such that $x$ is a point of continuity of the distribution function $x\mapsto \Prob(W\leq x)$, we have
$$\lim_{N\rightarrow +\infty} \Prob_N(W_N\leq x)=\Prob(Z\leq x). $$
It is also equivalent to say that for all {\it bounded continuous} function $f:\R\rightarrow \R$, we have
$$\lim_{N\rightarrow \infty} \Exp_N( f(W_N))=\Exp(f(W)).$$
Note that by (\ref{uniqueness1}), it is clear that any distribution of the type
$$Z=D_0+\sum_{p=1}^m\alpha_p Z_p, $$
where $D_0$ is a constant, $\alpha_1,\ldots,\alpha_p$ are real constant coefficients and $Z_1,\ldots,Z_p$ are independent Poisson variables with parameter $1$, is uniquely determined by its moments.
Indeed we have for all $t>0$, and setting $Y=\sum_{p=1}^m\alpha_p Z_p$, 
$$\sum_k \frac{t^k}{k!}\vert Y+D_0\vert^k\leq \sum_k \frac{t^k}{k!}( \vert Y\vert +\vert D_0\vert)^k=e^{t\vert Y\vert}e^{t\vert D_0\vert},$$
and we can use the fact that $D_0$ is a constant and $\vert Y\vert \leq \sum_p \vert \alpha_p\vert Z_p$ to justify, using (\ref{uniqueness1}), that for any $t>0$, 
$$\Exp\left ( \sum_k \frac{t^k}{k!}\vert Y+D_0\vert^k\right)\leq \Exp( e^{t\vert Y\vert})e^{t\vert D_0\vert}<\infty.$$
\subsection{Convergence of moments and end of proof of Proposition \ref{Moments}}
Our analysis will rely on the following facts, which are essentially due to Nica \cite{Nica1} (see also the related paper \cite{DJPP}) for free groups. We also point out the work of Y. Moaz \cite{Moaz} for the case of surface groups. A general statement
for more general groups is also given in \cite{PZ1}. We state below the bare minimum required for our needs.
\begin{propo}
\label{decorrelation1}
Using the preceding notations, the following holds.
\begin{enumerate}
\item Given a primitive word $\gamma \in \mathbb{F}_0^d$, we have for all $k\in \N$,
$$\lim_{N\rightarrow +\infty} \Exp_N( (F_N(\gamma))^k)=B_k,$$
where $B_k$ is as above the $k$-th Bell number.
\item Consider a finite set of distincts primitive words $(\gamma_j)_{j\in \mathcal{F}}\in  \mathbb{F}_0^d$ such that for all $i\neq j$ we have $\gamma_i\neq \gamma_j^{-1}$.
Fix a set of non-negative integers $(k_j)_{j\in \mathcal{F}}$. Then we have as $N$ goes to infinity,
$$\Exp\left ( \prod_{j\in \mathcal{F}} (F_N(\gamma_j))^{k_j}\right)=\prod_{j\in \mathcal{F}}\Exp\left ( (F_N(\gamma_j))^{k_j}\right)+O\left (\frac{1}{N}\right).$$
\end{enumerate}
\end{propo}
Clearly by the moment method, $(1)$ implies convergence in distribution of $F_N(\gamma)$ to a Poisson variable with parameter $1$, which is the main result in \cite{Nica1}. Statement
$(2)$ is an asymptotic multiple decorrelation result for distincts words in $\mathbb{F}_0^d$ which extends Proposition \ref{Stat1} $(2)$.

Let us recall that we have set
$$X_N=2Y_N+D_N, $$
where
$$D_N=-2\sum_{1\leq i<j\leq d}\Re(H_{i,j})+O\left( \frac{1}{N}\right) $$
and
$$Y_N:= \sum_{1\leq i<j\leq d}F_N(a_i^{-1}a_j)\Re(H_{i,j}). $$
We first write for $k$ fixed, by the multinomial formula
$$\Exp_N \left ( Y_N^k\right)=\mathlarger{\mathlarger{\mathlarger{\sum}}}_{\sum_{i<j}\ell_{i,j}=k}\frac{k!}{\prod_{i<j} \ell_{i,j}!} \left (\prod_{i<j} (\Re(H_{i,j}))^{\ell_{i,j}}\right) \Exp_N\left ( \prod_{i<j}(F_N(a_i^{-1}a_j))^{\ell_{i,j}}\right).$$
Since the summation is over indices $(i,j)$ such that $1\leq i<j\leq d$, whenever $(i_1,j_1)\neq(i_2,j_2)$, we can't have $a_{i_1}^{-1}a_{i_2}=(a_{j_1}^{-1}a_{j_2})^{-1}$, hence we can actually apply Proposition \ref{decorrelation1} $(2)$ which yields 
$$\Exp_N \left ( Y_N^k\right)=$$
$$\mathlarger{\mathlarger{\mathlarger{\sum}}}_{\sum_{i<j}\ell_{i,j}=k}\frac{k!}{\prod_{i<j} \ell_{i,j}!} \left (\prod_{i<j} (\Re(H_{i,j}))^{\ell_{i,j}}\right) \left ( \prod_{i<j}\Exp_N\left ((F_N(a_i^{-1}a_j))^{\ell_{i,j}}\right)\right)+O\left (\frac{1}{N}\right).$$
Applying Proposition \ref{decorrelation1} $(1)$ for each term $\Exp_N\left ((F_N(a_i^{-1}a_j))^{\ell_{i,j}}\right),$ individually we get immediately ($k$ being fixed),
$$\lim_{N\rightarrow+\infty}  \Exp_N \left ( Y_N^k\right)=\mathlarger{\mathlarger{\mathlarger{\sum}}}_{\sum_{i<j}\ell_{i,j}=k}\frac{k!}{\prod_{i<j} \ell_{i,j}!} \left (\prod_{i<j} (\Re(H_{i,j}))^{\ell_{i,j}}\right) \left ( \prod_{i<j}B_{\ell_{i,j}}\right).$$
We recognize from (\ref{momentform1}) the formula for $\Exp(Z^k)$ where $Z$ is the random variable
$$Z=\sum_{i<j} \Re(H_{i,j}) Z_{i,j},$$
with $Z_{i,j}$ being independent Poisson variables with parameter $1$. To prove part $(2)$ and $(3)$ of Proposition \ref{Moments}, we just observe that for all $k$ we have
$$\Exp_N\left(X_N^k\right)=\Exp_N\left( (2Y_N+D_N)^k\right)=\sum_{\ell=0}^k \binom{k}{\ell}2^\ell\Exp_N\left(Y_N^\ell\right) (D_N)^{k-\ell},$$
so that
$$L_k:=\lim_{N\rightarrow +\infty}\Exp_N\left(X_N^k\right)= \sum_{\ell=0}^k \binom{k}{\ell}2^\ell\Exp\left(Z^\ell\right) (D_0)^{k-\ell}=\Exp\left( (2Z+D_0)^k\right),$$
with
$$D_0=\lim_{N\rightarrow =\infty}D_N=-2\sum_{1\leq i<j\leq d}\Re(H_{i,j}).$$
Since $2Z+D_0$ is uniquely determined by its moments, the method of moments implies convergence in distribution of $X_N$ to $2Z+D_0$, thus finishing the proof of Proposition \ref{Moments}.

For the record, we have already computed 
$$L_2=4\sum_{i_1<i_2}  (\Re(H_{i_1,i_2}))^2.$$
Using the above formula (or alternative formulas for moments of centered Poisson distributions) we can also compute:
$$L_3=8\sum_{i_1<i_2}  (\Re(H_{i_1,i_2}))^3,\ L_4=16 \sum_{i_1<i_2}  (\Re(H_{i_1,i_2}))^4+48\left (  \sum_{i_1<i_2}  (\Re(H_{i_1,i_2}))^2\right)^2.$$

\section{The polynomial method and trace asymptotics}
This section is devoted to the proof of Theorem \ref{poly1}. We will rely on an adaptation (to the infinite dimensional case) of one the main inequalities that appear in the so-called "polynomial method" as used in \cite{CGTVH} for free groups. We start with some functional analytic preliminaries.
\subsection{Exponential classes}
Let $\mathcal{H}$ be a complex separable Hilbert space. We denote by $B(\mathcal{H})$ the $C^*$-algebra of bounded linear operator on $\mathcal{H}$. We denote by $S_1(\mathcal{H})\subset B(\mathcal{H})$ the 
subspace of {\it trace class operators}. Let  $a>0$. We define the exponential class $E(a,\mathcal{H})$ by
$$E(a,\mathcal{H})=\left\{ A \in S_1(\mathcal{H})\ :\ \sup_{j\in \N} \mu_j(A)e^{aj}<\infty \right\}.$$
Note that $E(a,\mathcal{H})$ is not a linear subspace of $S_1(\mathcal{H})$. Exponential classes have nevertheless some interesting stability properties that we summarize below.
\begin{propo}
\label{Exp1}
\begin{enumerate}
 \item If $A\in E(a,\mathcal{H})$ and $B\in B(\mathcal{H})$, then both $AB$ and $BA$ belong to $E(a,\mathcal{H})$.
 \item If $A\in E(a,\mathcal{H})$, so does its adjoint $A^*$.
 \item Assume that $A_k \in E(a_k,\mathcal{H})$ for $k=1,\ldots,r$, then we have
 $$\sum_{k=1}^r A_k\in E(a',\mathcal{H}),\ \mathrm{with}\ a'=\frac{1}{\sum_{k=1}^r a_k^{-1}}.$$
\end{enumerate}
\end{propo}
Fact $(1)$ follows from $\mu_j(AB)\leq \Vert B \Vert \mu_j(A)$, see Simon \cite{BSimon}, Theorem 1.5. Claim $(2)$ is due to the fact that $A$ and $A^*$ have the same singular value sequence. Indeed,
for all $\lambda\neq 0$ the map $A$ is an isomorphism
$$A:\mathrm{Ker}(A^*A-\lambda)\rightarrow \mathrm{Ker}(AA^*-\lambda). $$
For a proof of $(3)$, based on rearrangement inequalities, see Bandtlow in \cite{Bandtlow1}. 

Since we have the well known fact that for compact operators $A$, singular values coincide with approximation numbers i.e.,
$$\mu_j(A)=\inf\{ \Vert A-F\Vert\ :\ F\in B(\mathcal{H}),\ \mathrm{rank}(B)<j\}, $$
it follows that if $A\in E(a,\mathcal{H})$, for all $\epsilon>0$, one can find a sequence $(F_j)_j$ of finite rank operators with $\mathrm{rank}(F_j)\leq j$ such that
$$\Vert A-F_j\Vert\leq Ce^{-(a-\epsilon)j},$$
for some constant $C>0$. We will need a slightly stronger approximation result, for the trace norm.
\begin{propo}
\label{Approxtr}
Let $A\in E(a,\mathcal{H})$, then there exist $C,\widetilde{a}>0$ and a sequence of finite rank operators $(A_k)_k$ with $\mathrm{rank}(A_k)\leq k$ such that for all $k$ we have
$$\Vert A-A_k\Vert_{1}\leq C e^{-\widetilde{a}k}.$$
Note that since $\Vert A-A_k\Vert=\mu_0(A-A_k)\leq \Vert A-A_k\Vert_{1}$, the same estimate is also implied for the operator norm.
\end{propo}
\noindent {\it Proof}. By writing
$$A=\frac{A+A^*}{2}+i\left ( \frac{A-A^*}{2i}\right),$$ 
we have $U=\frac{A+A^*}{2}$ and $S=\frac{A-A^*}{2i}$ which are both self-adjoint and belong to $E(a/2,\mathcal{H})$ by Proposition \ref{Exp1}. We are therefore left with the problem of approximating a compact {\it self-adjoint} operator, say $B$, belonging to some exponential class $E(a,\mathcal{H})$. By the usual spectral theorem, denoting by
$(\lambda_j)_j$ the eigenvalue sequence of $B$, ordered so that $\vert \lambda_1\vert\geq \vert \lambda_2\vert \geq \ldots \geq \vert \lambda_j\vert\geq \ldots$, there exists a Hilbert basis $(e_j)_j$ of $\overline{\mathrm{Range}(B)}$ so that for all $j\geq 1$, we have $B(e_j)=\lambda_j e_j$. Define then $B_k$ to be given for all $f\in \mathcal{H}$ by
$$B_k(f):=\sum_{j=1}^k \lambda_j \langle f,e_j\rangle_{\mathcal{H}}e_j.$$
Clearly $\mathrm{rank}(B_k)\leq k$ and since $B-B_k$ is self-adjoint we have
$$\Vert B-B_k\Vert_1=\sum_{j=k+1}^\infty \vert \lambda_j\vert,$$
with $\vert \lambda_j\vert=\mu_j(B)$. Since we have the upper bound $\mu_j(B)\leq Ce^{-aj}$, we deduce that
$$\Vert B-B_k\Vert_1 \leq C \sum_{j=k+1}^\infty e^{-aj} =C\frac{e^{-a(k+1)}}{1-e^{-a}},$$
and the proof is done. $\square$

A fundamental fact for us, which we have already proved in disguise, is the following.
\begin{propo}
 Let $\gamma:\Omega_0\rightarrow \Omega_0$  be  a holomorphic contraction such that $$\overline{\gamma(\Omega_0)}\subset \Omega_0,$$ then the composition operator $T_\gamma$ acting
 on $H^2(\Omega_0)$ is in an exponential class.
\end{propo}
\noindent {\it Outline of proof}. By the min-max theorem applied to $\sqrt{T_\gamma^*T_\gamma}$, we have 
$$\mu_j(T_\gamma)\leq \sum_{k=j}^\infty \Vert T_\gamma(e_k)\Vert_{H^2(\Omega_0)},$$
for all Hilbert basis $(e_k)$ of $H^2(\Omega_0)$. By using an "explicit" basis, exactly as in the proof of Proposition \ref{tracest1}, we obtain existence of $C>0$ and $0<\rho<1$ such that
for all $k$, $$\Vert T_\gamma(e_k)\Vert_{H^2(\Omega_0)}\leq C\rho^k, $$
and the proof is complete since we get immediately
$$\mu_j(T_\gamma)\leq C \frac{\rho^j}{1-\rho}.$$
$\square$

\subsection{Tracial states and group algebras} 
Let $\C[\mathbb{F}^d]$ be the group algebra of the free group $\mathbb{F}^d$, that is the set of formal sums 
$$x=\sum_{\gamma \in \mathbb{F}^d}a_\gamma.\gamma$$
where $a_\gamma \in \C$ and $\gamma \mapsto a_\gamma$ has finite support in $\mathbb{F}^d$. We recall that $\lambda: \mathbb{F}^d\rightarrow U(\mathbb{F}^d)$ is the left regular representation. We will denote by
$\lambda(x):=\sum_\gamma a_\gamma \lambda(\gamma)$ which is a bounded operator acting on $\ell^2(\mathbb{F}^d)$. 

Recall also that $\phi_N: \mathbb{F}^d\rightarrow S_N$ is a random permutation representation of the free group $\mathbb{F}^d$. Each permutation $\phi_N(\gamma)$ acts on $\ell^2(\{1,\ldots,N\})$ and we denote
$\mathcal{V}_N$ the orthogonal subspace to constant functions in $\ell^2(\{1,\ldots,N\})$. The representation $\phi_N:\mathbb{F}^d\rightarrow U(\mathcal{V}_N)$ is then denoted by $\pi_N$. For all 
$x=\sum_{\gamma \in \mathbb{F}^d}a_\gamma.\gamma\in \C[\mathbb{F}^d]$ we will use the notation $\pi_N(x)=\sum_\gamma a_\gamma \pi_N(\gamma)$.

Of primary interest for us will be the group algebras $\mathcal{M}_D(\C)[\mathbb{F}^d]$ and $S_1(\mathcal{H})[\mathbb{F}^d]$ which consists of formal sums
$$\mathscr{X}=\sum_\gamma A_\gamma.\gamma, $$
where $A_\gamma$ are respectively $D\times D$ complex matrices and trace class operators on a Hilbert space $\mathcal{H}$. We say that such an element $\mathscr{X}$ 
is {\it self-adjoint} iff
$A_{\gamma^{-1}}=A_\gamma^*$ for all $\gamma \in \mathbb{F}^d$.

In Both cases we define $\pi_N(\mathscr{X})$ and
$\lambda(\mathscr{X})$ by
$$\pi_N(\mathscr{X})=\sum_\gamma A_\gamma \otimes \pi_N(\gamma),\  \lambda(\mathscr{X})=\sum_\gamma A_\gamma \otimes \lambda(\gamma),$$
acting respectively on the tensor products $\C^D\otimes \mathcal{V}_N$ and $\C^D\otimes \ell^2(\mathbb{F}^d)$ or $\mathcal{H}\otimes \mathcal{V}_N$ and $\mathcal{H}\otimes \ell^2(\mathbb{F}^d)$.
We then define the "tracial state" $\tau$ to be 
$$\tau(\lambda(\mathscr{X})):=\sum_\gamma \mathrm{Tr}(A_\gamma)\langle \lambda(\gamma)\chi_e,\chi_e\rangle_{\ell^2(\mathbb{F}^d)}=\mathrm{Tr}(A_{Id}),$$
where $\chi_e(\gamma)=1$ if $\gamma=Id$, and $0$ elsewhere.

Classical results in free probability (namely asymptotic freeness of independent random permutation matrices) imply that for all $\mathscr{X}\in \mathcal{M}_D(\C)[\mathbb{F}^d]$, with $D$ fixed, we have
$$\lim_{N\rightarrow +\infty} \frac{\Exp_N(\mathrm{Tr}(\pi_N(\mathscr{X})))}{N}=\tau(\lambda(\mathscr{X})).$$
This result was first shown by Nica in \cite{Nica2}. The main idea of the polynomial method as explained in \cite{CGTVH} is to obtain similar asymptotics for self-adjoint elements $\mathscr{X}$, when $\pi_N(\mathscr{X})$ is replaced by $\varphi(\pi_N(\mathscr{X}))$ where $\varphi \in C^\infty(\R)$ is an arbitrary smooth, real valued, function. Achieving this task implies a soft proof of "strong asymptotic freeness", also called strong convergence in a wider context, which was first proved for random permutation matrices by Bordenave and Collins in \cite{BC1}.

The first basic question is to understand what $\tau(\varphi(\lambda(\mathscr{X})))$ means when $\varphi$ is a smooth function. Assume first that
$\varphi$ is a real  polynomial and write
$$\varphi( \lambda(\mathscr{X}))=\sum_\gamma B_\gamma\otimes \lambda(\gamma),$$
with $B_\gamma \in \mathcal{M}_D(\C)$. We have by definition
$$\tau(\varphi(\lambda(\mathscr{X})))=\sum_{\gamma \in \Gamma} \mathrm{Tr}(B_\gamma) \langle \lambda(\gamma) \chi_e,\chi_e\rangle_{\ell^2(\mathbb{F}^d)} ,$$
which by picking an orthonormal basis of $\C^D$ can be written as
$$\tau(\varphi(\lambda(\mathscr{X})))=\sum_{i=1}^D\sum_{\gamma \in \Gamma} \langle B_\gamma e_i,e_i \rangle_{\C^D} \langle \lambda(\gamma) \chi_e,\chi_e\rangle_{\ell^2(\mathbb{F}^d)}$$
$$=\sum_{i=1}^D \langle \varphi( \lambda(\mathscr{X})) e_i\otimes \chi_e,e_i\otimes \chi_e \rangle_{\C^D\otimes \ell^2(\mathbb{F}^d)}.$$
By Cauchy-Schwarz we simply get 
$$\vert \tau(\varphi(\lambda(\mathscr{X})))\vert \leq D \Vert  \varphi(\lambda(\mathscr{X}))\Vert.$$
Since $\varphi(\lambda(\mathscr{X}))$ is a bounded self-adjoint operator, its norms equals its spectral radius and applying the spectral mapping theorem gives
$$\vert \tau(\varphi(\lambda(\mathscr{X})))\vert \leq D \sup_{z\in \mathrm{Sp}(\lambda(\mathscr{X}))} \vert \varphi(z)\vert.$$
The linear functional $\varphi \mapsto \tau(\varphi(\lambda(\mathscr{X})))$ is therefore continuous on the Banach space $C^0(\mathrm{Sp}(\lambda(\mathscr{X})))$ and therefore extends by density from polynomials to any continuous function. Since it is also positive, Riesz theorem says that there exists a finite Radon measure $\mu$, supported in $\mathrm{Sp}(\lambda(\mathscr{X}))$ such that
$$\tau(\varphi(\lambda(\mathscr{X})))=\int \varphi d\mu. $$
We need an infinite-dimensional extension of this line of ideas.
We first make the following basic remark.
\begin{propo}
\label{trace1}
The tracial state $\tau$ is "faithfull" on $S_1(\mathcal{H})[\mathbb{F}^d]$ i.e. for all 
$\mathscr{Y} \in S_1(\mathcal{H})[\mathbb{F}^d]$ with
$$\mathscr{Y}=\sum_\gamma C_\gamma.\gamma,$$
we have for any Hilbert basis $(e_i)_{i\in \N}$ of $\mathcal{H}$,
$$\tau ((\lambda(\mathscr{Y}))^*\lambda( \mathscr{Y}))=\sum_\gamma  \mathrm{Tr}(C_\gamma^*C_{\gamma})=\sum_{i=0}^\infty \Vert \lambda(\mathscr{Y})e_i\otimes \chi_e \Vert^2.$$
\end{propo}
\noindent {\it Proof}. Write
$$\mathscr{Y}=\sum_\gamma C_\gamma.\gamma,$$
so that
$$(\lambda(\mathscr{Y}))^*\lambda( \mathscr{Y})=\sum_{\gamma,\gamma'} C_\gamma^*C_{\gamma'}\otimes \lambda( \gamma^{-1}\gamma'), $$
and therefore
$$\tau ((\lambda(\mathscr{Y}))^*\lambda( \mathscr{Y}))=\sum_{\gamma,\gamma'} \mathrm{Tr}(C_\gamma^*C_{\gamma'})\langle \lambda( \gamma^{-1}\gamma')\chi_e,\chi_e\rangle
=\sum_\gamma  \mathrm{Tr}(C_\gamma^*C_{\gamma}).$$
On the other hand we have by definition of the trace,
$$\mathrm{Tr}(C_\gamma^*C_{\gamma'})=\sum_i \langle C_\gamma^*C_{\gamma'}e_i,e_i\rangle, $$
which also gives
$$\tau ((\lambda(\mathscr{Y}))^*\lambda( \mathscr{Y}))=\sum_i \sum_{\gamma,\gamma'}\langle [C_\gamma^*\otimes \lambda(\gamma^{-1})][  
C_{\gamma'}\otimes \lambda(\gamma')]e_i\otimes \chi_e,e_i\otimes \chi_e\rangle$$
$$=\sum_{i=0}^\infty \Vert \lambda(\mathscr{Y})e_i\otimes \chi_e \Vert^2,$$
and the result is proved. $\square$
 
\begin{propo}
\label{tracial}
Let $P$ be a real polynomial, and assume that 
$$\mathscr{X}=\sum_\gamma A_\gamma.\gamma\in S_1(\mathcal{H})[\mathbb{F}^d]$$ is self-adjoint and positive. Then we have 
$$\vert  \tau( \lambda(\mathscr{X}) P( \lambda(\mathscr{X})))\vert\leq \mathrm{Tr} \left (\sqrt{B}\right) \sup_{z\in \mathrm{Sp}(\lambda(\mathscr{X}))} \vert P(z)\vert,$$
where $B=\sum_{\gamma} A_\gamma^*A_\gamma$. Moreover, the functional 
$$P\mapsto   \tau( \lambda(\mathscr{X}) P( \lambda(\mathscr{X})))$$
extends to a finite positive Radon measure $\mu_{\mathscr{X}}$ whose support is included in $\mathrm{Sp}(\lambda(\mathscr{X}))$ and contains
$\mathrm{Sp}(\lambda(\mathscr{X}))\setminus\{0\}$.
\end{propo}
{\it Proof}. We first need to justify that $\sqrt{B}$ is indeed trace class. Let $r$ denote the cardinality of the support of $\gamma \mapsto A_\gamma$. 
By Fan inequality, see Simon \cite{BSimon}, Theorem 1.7, we have
$$\mu_{rj+1}(B)\leq \sum_\gamma \mu_j(A_\gamma^*A_\gamma)=\sum_\gamma \mu_j^2(A_\gamma).$$
Since $B$ is self-adjoint positive we have actually
$$\lambda_{rj+1}(B)=\mu_{rj+1}(B)\leq  \sum_\gamma \mu_j^2(A_\gamma)\leq \left ( \sum_\gamma \mu_j(A_\gamma) \right)^2.$$
By spectral mapping and Lidskii's theorem, we have therefore
$$\mathrm{Tr}(\sqrt{B})=\sum_\ell \sqrt{\lambda_\ell(B)}, $$
which implies (recall that the sequence of eigenvalues $(\lambda_\ell)_\ell$ is decreasing) 
$$\sqrt{\lambda_\ell}\leq \sqrt{\lambda_{r[\ell/r]+1}}\leq  \ \sum_\gamma \mu_{[\ell/r]}(A_\gamma) $$
and therefore $\sum_\ell \sqrt{\lambda_\ell(B)}<\infty$ since each series $\sum_j \mu_j(A_\gamma)$ is finite.
Note that we have actually obtained the estimate
\begin{equation}
\label{sqrest}
\mathrm{Tr}(\sqrt{B})\leq \sum_\gamma \Vert A_\gamma\Vert_1
\end{equation}
which will be helpful further on.
We now refocus on the calculation of $\tau( \lambda(\mathscr{X}) P( \lambda(\mathscr{X})))$, trying to mimic the finite-dimensional case.
Let $P\in \R[X]$, we have 
\begin{equation}
\label{stateform1}
\tau(\lambda(\mathscr{X}) P( \lambda(\mathscr{X})))=
\sum_{i=0}^\infty \langle \lambda(\mathscr{X}) P( \lambda(\mathscr{X}))e_i\otimes \chi_e, e_i\otimes \chi_e \rangle_{\mathcal{H}\otimes \ell^2(\mathbb{F}^d)},
\end{equation}
for any Hilbert basis $(e_i)_{i\in \N}$ of $\mathcal{H}$. Therefore by Cauchy-Schwarz we get
$$\vert  \tau(\lambda(\mathscr{X}) P( \lambda(\mathscr{X})))\vert \leq \Vert P(\lambda(\mathscr{X}))\Vert 
\sum_{i=0}^\infty \Vert \lambda(\mathscr{X})) e_i\otimes \chi_e \Vert$$
Since $\lambda(\mathscr{X}))$ is a bounded self-adjoint operator, we have actually by the spectral mapping theorem
$$\Vert P(\lambda(\mathscr{X}))\Vert =\sup_{z\in \mathrm{Sp}(\lambda(\mathscr{X})))} \vert P(z)\vert.$$
We also observe that
$$\sum_{i=0}^\infty \Vert \lambda(\mathscr{X})) e_i\otimes \chi_e \Vert=
\sum_i \left (\langle \lambda^*(\mathscr{X})) \lambda(\mathscr{X})) e_i\otimes \chi_e, e_i\otimes \chi_e \rangle \right )^{1/2}$$
$$=\sum_i \left (\sum_\gamma \langle A_\gamma^*A_\gamma e_i, e_i\rangle_{\mathcal{H}} \right )^{1/2}=\sum_i \left (\langle B e_i, e_i\rangle_{\mathcal{H}} \right )^{1/2}$$
But $B$ is self-adjoint positive, trace class, and thus by the above discussion we have exactly (by taking $(e_i)$ to be an orthonormal basis diagonalizing $B$ )
$$\sum_i \left (\langle B e_i, e_i\rangle_{\mathcal{H}} \right )^{1/2}=\mathrm{Tr}(\sqrt{B}).$$
We have obtained the inequality
 $$\vert  \tau( \lambda(\mathscr{X}) P( \lambda(\mathscr{X})))\vert\leq \mathrm{Tr} \left (\sqrt{B}\right) \sup_{z\in \mathrm{Sp}(\lambda(\mathscr{X}))} \vert P(z)\vert,$$
 The linear functional
 $$P\mapsto  \tau( \lambda(\mathscr{X}) P( \lambda(\mathscr{X})))$$
 is therefore continous on $\R[X]\cap C^0(\mathrm{Sp}(\lambda(\mathscr{X})))$, and extends by Stone-Weierstrass to a continous linear functional on $C^0(\mathrm{Sp}(\lambda(\mathscr{X})))$.
 By Riesz theorem, there exists a finite (a priori signed) Radon measure $\mu_{\mathscr{X}}$, 
 supported inside $\mathrm{Sp}(\lambda(\mathscr{X}))$, such that for all continuous function $\varphi$
 $$\tau( \lambda(\mathscr{X}) \varphi( \lambda(\mathscr{X})))=\int \varphi d \mu_{\mathscr{X}}.$$
 We also remark that if $P\in \R[X]$ is such that $XP(X)$ is positive on $\mathrm{Sp}(\lambda(\mathscr{X}))$, then the operator
 $\lambda(\mathscr{X}) P( \lambda(\mathscr{X})) $ is also positive and thus by (\ref{stateform1})
 $$\int P d \mu_{\mathscr{X}}\geq 0.$$
 Since we have assumed that $\lambda(\mathscr{X})$ is a positive operator, by a easy density argument, $\mu_{\mathscr{X}}$ is therefore a positive measure.
 It remains to prove that the support of $\mu_{\mathscr{X}}$ contains $\mathrm{Sp}(\lambda(\mathscr{X}))\setminus \{0\}$. We  assume $\lambda(\mathscr{X})\geq 0$. Let $x_0\in \mathrm{Sp}(\lambda(\mathscr{X}))$ with $x_0>0$.
 Assume that there exists $\epsilon>0$ such that $0<x_0-\epsilon$ and 
 $$\mu_{\mathscr{X}}((x_0-\epsilon,x_0+\epsilon))=0.$$
 We pick $\varphi \in C_0^\infty(\R)$ to be a non-negative function such that
 $\varphi(x_0)=1$, and $\varphi \leq 1$ with $\mathrm{Supp}(\varphi)\subset (x_0-\epsilon,x_0+\epsilon)$.
 Write $\varphi(x)=x\psi(x)$, where $\psi$ is also obviously smooth. Note that $\varphi(\lambda(\mathscr{X})))$, defined by functional calculus, is a non zero operator since by spectral mapping and self-adjointness
 $\Vert  \varphi(\lambda(\mathscr{X})))\Vert\geq 1$.
 On the other hand we have
 $$\tau( \varphi^2(\lambda(\mathscr{X})))=\tau(\lambda^2(\mathscr{X})\psi^2(\lambda(\mathscr{X})))=\int u\psi^2(u)d\mu_{\mathscr{X}}(u) $$
 $$=\int \varphi^2(u)\frac{d\mu_{\mathscr{X}}(u)}{u}\leq \int_{(x_0-\epsilon,x_0+\epsilon)}\frac{d\mu_{\mathscr{X}}(u)}{u}=0.$$
 To get a contradiction, we need to show that $\tau( \varphi^2(\lambda(\mathscr{X})))$ is indeed positive. Note that $\varphi(\lambda(\mathscr{X}))$ is a priori no longer in
 $S_1(\mathcal{H})[\mathbb{F}^d]$ but rather is a limit in operator norm of a sequence of self-adjoint operators
 $$C_n=\lambda(\mathscr{X})P_n(\lambda(\mathscr{X})),$$
 where $P_n$ are real polynomials. We have therefore 
 $$0= \tau( \varphi^2(\lambda(\mathscr{X})))=\lim_{n\rightarrow \infty} \sum_i \Vert \lambda(\mathscr{X})P_n(\lambda(\mathscr{X}))e_i\otimes \chi_e \Vert^2,$$
 in particular we get for all $i\in \N$,
 $$ \varphi(\lambda(\mathscr{X}))(e_i\otimes \chi_e) =0.$$
 The end of the proof is then a standard trick, see for example in \cite{NicaSpeicher}, Chapter 3. 
 Let $\delta_h:\ell^2(\mathbb{F}^d)\rightarrow \ell^2(\mathbb{F}^d)$, with $h\in \mathbb{F}^d$ be defined by
 $$\delta_h(f)(g):=f(gh^{-1}),$$
 then for any $\gamma,h\in \mathbb{F}^d$, $\lambda(\gamma)$ and $\delta_h$ commute. As a consequence $Id\otimes \delta_h$ commutes with any element of the type $\lambda(\mathscr{Y})$ with $\mathscr{Y}=\in S_1(\mathcal{H})[\mathbb{F}^d]$. As a limit in operator norm of such elements, $ \varphi(\lambda(\mathscr{X}))$ must commute with $Id\otimes \delta_h$ also.
 Therefore we have for any $h$,
 $$ \varphi(\lambda(\mathscr{X}))(e_i\otimes \chi_h)=\varphi(\lambda(\mathscr{X}))(Id\otimes \delta_h)(e_i\otimes \chi_e)=(Id\otimes \delta_h)\varphi(\lambda(\mathscr{X}))(e_i\otimes \chi_e)=0,$$
 and hence $\varphi(\lambda(\mathscr{X}))$ must be $0$ since it is vanishing on a Hilbert basis of $\mathcal{H}\otimes \ell^2(\mathbb{F}^d)$.
 A contradiction. We have therefore shown that $$\mathrm{Sp}(\lambda(\mathscr{X}))\setminus \{0\}\subset \mathrm{supp}(\mu_{\mathscr{X}})\subset \mathrm{Sp}(\lambda(\mathscr{X})).$$ 
The proof is done. $\square$

Note that if we know that $0$ is not isolated in the spectrum of $\lambda(\mathscr{X})$, then we have actually $\mathrm{supp}(\mu_{\mathscr{X}})=\mathrm{Sp}(\lambda(\mathscr{X}))$.
Under the above hypotheses, we always have $0=\inf \mathrm{Sp}(\lambda(\mathscr{X}))$. Indeed, we have by direct computation
$$\langle \lambda(\mathscr{X})f\otimes \chi_g,f\otimes \chi_g \rangle=\langle A_e f,f\rangle_{\mathcal{H}},$$
and since $A_e$ is assumed to be a self-adjoint positive compact operator, $0$ is in its spectrum which implies
$$\inf_{f\in \mathcal{H}, \atop \Vert f\Vert=1} \langle A_e f,f\rangle_{\mathcal{H}}=0.$$
However at this level of generality, we cannot rule out the fact that $0$ could be  an isolated point of the spectrum of $\lambda(\mathscr{X})$: for example if $\lambda(\mathscr{X})=A_e\otimes Id$, 
with $A$ a finite rank operator, then this is the case.

We also point out to the reader that the "tracial state" $\tau$, unlike in the finite dimensional case of $\mathcal{M}_D(\C)[\mathbb{F}^d]$, does not make sense in the operator norm closure of
$S_1(\mathcal{H})[\mathbb{F}^d]$, simply because operator norm limits of trace class operators are in general not trace class. 

We end this subsection by a key inequality.
\begin{propo}
\label{tracialest2}
 Let $\mathscr{X},\mathscr{Y}\in S_1(\mathcal{H})[\mathbb{F}^d]$, with 
 $$\mathscr{X}=\sum_\gamma A_\gamma.\gamma,\  \mathscr{Y}=\sum_\gamma C_\gamma.\gamma.$$
 Let $\psi \in C_0^\infty(\R)$, then there exists $C_\psi>0$ such that
 $$\vert  \tau\left \{ \lambda(\mathscr{X}) \psi( \lambda(\mathscr{X}))\right \}-\tau\left \{ \lambda(\mathscr{Y}) \psi( \lambda(\mathscr{Y}))\right \}\vert \leq $$
 $$\mathrm{Tr}(\sqrt{B})C_\psi \Vert \lambda(\mathscr{X})-\lambda(\mathscr{Y}) \Vert+\left( \sup_{x\in \R} \vert \psi(x) \vert \right)\sum_\gamma \Vert A_\gamma-C_\gamma \Vert_1,$$
 where as before, $B=\sum_\gamma A_\gamma^* A_\gamma$, and $C_\psi$ is a constant involving only $\psi$, see below.
\end{propo}
\noindent{\it Proof}. Assume first that $\psi$ is a real polynomial. We write
$$ \tau\left \{ \lambda(\mathscr{X}) \psi( \lambda(\mathscr{X}))\right \}-\tau\left \{ \lambda(\mathscr{Y}) \psi( \lambda(\mathscr{Y}))\right \}= $$
$$\tau\left \{ (\psi(\lambda(\mathscr{X}))-\psi(\lambda(\mathscr{Y})))\lambda(\mathscr{X})\right \} +\tau \left \{(\psi(\lambda(\mathscr{Y}))(\lambda(\mathscr{X}))-\lambda(\mathscr{Y})) ) \right \}.$$
Proceeding as in the preceding proof, we get
$$\vert \tau\left \{ (\psi(\lambda(\mathscr{X}))-\psi(\lambda(\mathscr{Y})))\lambda(\mathscr{X})\right \} \vert \leq \Vert \psi(\lambda(\mathscr{X}))-\psi(\lambda(\mathscr{Y}))\Vert \mathrm{Tr}(\sqrt{B}),$$
while by (\ref{sqrest}) we obtain
 $$ \vert \tau \left \{(\psi(\lambda(\mathscr{Y}))(\lambda(\mathscr{X}))-\lambda(\mathscr{Y})) ) \right \} \vert \leq \left( \sup_{x\in \R} \vert \psi(x) \vert \right)\sum_\gamma \Vert A_\gamma-C_\gamma \Vert_1.$$
 The inequality extends by density of polynomials in $C^\infty_0(\R)$ (endowed with the Frechet topology) to any test function $\psi$. What remains to be controled is the term
 $$\Vert \psi(\lambda(\mathscr{X}))-\psi(\lambda(\mathscr{Y})) \Vert. $$
 There are several possible ways to bound it via the operator norm $\Vert \lambda(\mathscr{X})-\lambda(\mathscr{Y}) \Vert$, but none are straightforward. We give an outline below based on the
 Helffer-Sj\"ostrand formula from \cite{HS1,Davies}. In what follows we denote for all $x\in \R$, $\langle x\rangle:=\sqrt{1+x^2}$. Let $g_0 \in C_0^\infty(\R)$
 be such that $1\geq g_0\geq 0$, $g(x)\equiv 1$ if $\vert x\vert\leq 1$ and $\mathrm{supp}(g_0)\subset [-2,+2]$, and set 
 $$\sigma(x,y):=g_0\left ( \frac{y}{\langle x\rangle}\right).$$
 We assume that $\mathrm{supp}(\psi)\subset[-K,+K]$, with $K\geq 1$. We define an almost-analytic extension $\widetilde{\psi}$ of $\psi$ to $\C$ as follows (here $n$ is arbitrary). Set for all $z=x+iy\in \C$,
 $$\widetilde{\psi}(z):=\left ( \sum_{k=0}^n \frac{\psi^{(k)}(x)(iy)^k}{k!}\right)\sigma(x,y). $$
 Recall that we have 
 $$\frac{\partial }{\partial \overline{z}}=\frac{1}{2}\left ( \frac{\partial}{\partial x}+i  \frac{\partial}{\partial y}\right).$$
 A direct calculation gives 
 $$\frac{\partial \widetilde{\psi}}{\partial \overline{z}}=\frac{1}{2}\frac{\psi^{(n+1)}(x)(iy)^n}{n!}\sigma(x,y)+ \left ( \sum_{k=0}^n \frac{\psi^{(k)}(x)(iy)^k}{k!}\right) \frac{\partial \sigma}{\partial \overline{z}}
 (x,y).$$
 Set $$U_K:=\{z=x+iy\in \C\ :\ \vert x \vert \leq K\ \mathrm{and}\ \vert y \vert \leq 2\langle x\rangle\}$$
 and
  $$W_K:=\{z=x+iy\in \C\ :\ \vert x \vert \leq K\ \mathrm{and}\ \langle x \rangle \leq \vert y \vert \leq 2\langle x\rangle\}.$$
  Then we have
  $$\mathrm{supp}\left(\frac{\partial \widetilde{\psi}}{\partial \overline{z}}\right)\subset U_K,\ \mathrm{supp}\left(\frac{\partial \sigma}{\partial \overline{z}}\right)\subset W_K,$$
  and for all $z\in U_K$ we have the estimate
  $$\left \vert  \frac{\partial \widetilde{\psi}}{\partial \overline{z}}(z)\right \vert\leq \vert y \vert^n\sup_{\R}\vert \psi^{(n+1)}\vert +
  C_\sigma \left ( \sum_{k=0}^n \sup_{\R}\vert \psi^{(k)}\vert \right)(2\langle K \rangle)^n\chi_{W_K}(z), $$
  where we have set $C_\sigma:=\sup_{U_K} \left \vert  \frac{\partial \sigma}{\partial \overline{z}}\right \vert$. Setting $n=2$ and onserving that whenever $z\in W_K$ we have $\vert \Im(z)\vert \geq 1$, we get  
  \begin{equation}
  \label{Cpsiest}
  \int_{\C} \left\vert \frac{\partial \widetilde{\psi}}{\partial \overline{z}}\right\vert \vert \Im(z)\vert^{-2}dm(z) 
  \leq \sup_{\R} \vert \psi^{(3)}\vert \mathrm{Vol}(U_K) +C_\sigma\left ( \sum_{k=0}^2 \sup_{\R}\vert \psi^{(k)}\vert \right)(2\langle K \rangle)^2 \mathrm{Vol}(W_K).
  \end{equation}
  $$\leq M(\sigma,K)\Vert \psi \Vert_{C^3([-K,+K])},$$
  where $M(\sigma,K)>0$ depends only on $K$ and the choice of $\sigma$ i.e. $g_0$ and 
  $$\Vert \psi \Vert_{C^3([-K,+K])}=\max_{0\leq j\leq 3} \sup_{[-K,+K]}\vert \psi^{(j)}\vert.$$
  Let $T:\mathscr{H}\rightarrow \mathscr{H}$ be a bounded self-adjoint operator acting on a complex separable Hilbert space $\mathscr{H}$. Helffer-Sj\"ostrand formula, see \cite{Davies} chapter 2, says that we have for all $n\geq 1$, 
  $$\psi(T)=\frac{-1}{\pi}  \int_{\C} \frac{\partial \widetilde{\psi}}{\partial \overline{z}}R_T(z)dm(z),$$
  where $R_T(z):=(zId-T)^{-1}$ is the resolvent, which is analytic in $\C\setminus \mathrm{Sp}(T)$.
  As a consequence, if $T_1,T_2$ are bounded self-adjont operators we have by the resolvent formula
  $$\psi(T_1)-\psi(T_2)=\frac{-1}{\pi}  \int_{\C} \frac{\partial \widetilde{\psi}}{\partial \overline{z}}R_{T_1}(z)(T_1-T_2)R_{T_2}(z)dm(z)$$
  and thus using the fact that $\Vert R_T(z)\Vert\leq \vert \Im(z)\vert^{-1}$ for all $z\not \in \R$, we have
  $$\Vert \psi(T_1)-\psi(T_2)\Vert\leq \frac{1}{\pi}  \int_{\C} \left \vert \frac{\partial \widetilde{\psi}}{\partial \overline{z}}\right\vert \vert \Im(z)\vert^{-2}\Vert T_1-T_2\Vert dm(z).$$
  Inequality (\ref{Cpsiest}) gives 
  $$\Vert \psi(T_1)-\psi(T_2)\Vert\leq C_\psi \Vert T_1-T_2\Vert,$$
  with
  $$C_\psi:=\frac{1}{\pi}M(\sigma,K)\Vert \psi \Vert_{C^3([-K,+K])}.$$
  We record the above discussion in the following Lemma, to be used again in the next subsection.
  \begin{lem}
   \label{Lipschitzbound}
    Let $\psi\in C_0^\infty([-K,+K])$,
   then there exists $M_0(K)>0$, such that for all $T_1,T_2:\mathscr{H}\rightarrow \mathscr{H}$ bounded self-adjoints operators, we have in operator norm,
   $$\Vert \psi(T_1)-\psi(T_2)\Vert\leq M_0(K)\Vert \psi \Vert_{C^3([-K,+K])}\Vert T_1-T_2\Vert.$$
  \end{lem}
  For a more sophisticated proof based on the notion of "operator differentiable functions", we refer the reader to \cite{Pedersen}. Unlike what a naive approach would suggest, $C^1$-regularity of $\psi$ is not enough to guarantee that such an estimate holds, see in \cite{Pedersen} for counter examples, while  $C^2$ is enough. In our soft proof above, we essentially need $\psi$ to be $C^3$ which is non-optimal but enough for our purpose. The end of the proof of Proposition \ref{tracialest2} follows by applying directly the above bound to $T_1=\lambda(\mathscr{X})$ and $T_2=\lambda(\mathscr{Y}).$
 \subsection{Extending the polynomial method to $S_1(\mathcal{H})[\mathbb{F}^d]$}
 We start by another useful Lemma.
 \begin{lem}
 \label{tracenormbound}
  Let $\mathscr{X}\in S_1(\mathcal{H})[\mathbb{F}^d]$, say
  $$\mathscr{X}=\sum_\gamma A_\gamma.\gamma,$$
  then we have for all $N$,
  $$\Vert \pi_N(\mathscr{X})\Vert_1\leq N \sum_\gamma \Vert A_\gamma \Vert_1.$$
 \end{lem}
 {\it Proof.} Since we have
 $$\pi_N(\mathscr{X})=\sum_\gamma A_\gamma \otimes \pi_N(\gamma), $$
 it is enough to estimate $\Vert B\otimes \pi_N(\gamma)\Vert_1$ for any $B\in S_1(\mathcal{H})$.
 Picking an orthonormal Hilbert basis $(e_i\otimes f_j)_{i\in \N \atop 1\leq j \leq N-1}$ of $\mathcal{H}\otimes \mathcal{V}_N$,
 we have 
 $$(B\otimes \pi_N(\gamma))^*(B\otimes \pi_N(\gamma))=B^*B\otimes I_{N-1},$$
 and therefore
 $$\Vert B\otimes \pi_N(\gamma)\Vert_1=\sum_{i,j}\langle (\sqrt{B^*B}\otimes I_{N-1} )e_i\otimes f_j, e_i\otimes f_j \rangle_{\mathcal{H}\otimes \mathcal{V}_N}$$
 $$=\sum_{i,j}\langle \sqrt{B^*B}e_i, e_i \rangle\langle f_j,f_j\rangle=(N-1)\Vert B \Vert_1,$$
 and the proof is done. $\square$
 
 We will also use the following easy fact.
 \begin{lem}
 \label{tensornorm1}
  Let $A$ be in $B(\mathcal{H})$, then for all $\gamma \in \mathbb{F}^d$, we have
  $$\Vert A\otimes \pi_N(\gamma)\Vert\leq \Vert A \Vert.$$
  Similarly we have also for all $\gamma$,
  $$\Vert A\otimes \lambda(\gamma)\Vert\leq \Vert A \Vert.$$
   \end{lem}
   {\noindent{\it Proof}. Write for any $F\in \mathcal{H}\otimes \mathcal{V}_N$,
   $$F=\sum_{j=1}^{N-1} f_j\otimes e_j $$
   where $(e_j)$ is an orthonormal basis of $\mathcal{V}_N$ and $\Vert F \Vert^2=\sum_j \Vert f_j\Vert^2_{\mathcal{H}}$.
   By unitarity of $\pi_N(\gamma)$ we have immediately
   $$\Vert A\otimes \pi_N(\gamma) F\Vert^2=\left \Vert \sum_j A(f_j)\otimes \pi_N(\gamma)(e_j) \right \Vert^2=\sum_j \Vert A f_j\Vert^2\leq \Vert A \Vert^2 \Vert F\Vert^2.$$
   When $\mathcal{V}_N$ is replaced by $\ell^2(\mathbb{F}^d)$, we can proceed in the same way, replacing the finite basis of $\mathcal{V}_N$ by a Hilbert basis of $\ell^2(\mathbb{F}^d)$ and using unitarity of $\lambda(\gamma)$.
   The proof is done. $\square$
 
 Our aim is to prove the following fact, from which Theorem \ref{poly1} can be deduced readily by taking $\mathcal{H}=H^2(\Omega_0)$ and observing that $V_N\simeq \mathcal{H}\otimes \mathcal{V}_N$.
 The properties of exponential classes ensure that $\lt_N^*\lt_N$ can indeed be written as $\pi_N(\mathscr{X})$ for some self-adjoint positive element $\mathscr{X}\in S_1(\mathcal{H})[{\mathbb F}^d]$, whose "coefficients" are all in the same exponential class. 
 Applying the next result below combined with Proposition \ref{tracial} ends the proof.
 \begin{thm} 
 \label{main1}
 Assume that $\mathscr{X}=\sum_\gamma A_\gamma.\gamma\in S_1(\mathcal{H})[\mathbb{F}^d]$ is a self-adjoint positive element. Assume in addition that all coefficients $A_\gamma$ belong to some exponential class $E(a_0,\mathcal{H})$, for some $a_0>0$. Then there exists $C_{\mathscr{X}},K_{\mathscr{X}}>0$  independent of $\psi$ such that for all $\psi\in C_0^\infty(\R)$, we have for all $N$ large
 $$\left \vert \frac{\Exp_N\left (\mathrm{Tr}(\pi_N(\mathscr{X})\psi(\pi_N(\mathscr{X})))\right)}{N}-\tau(\lambda(\mathscr{X})\psi(\lambda(\mathscr{X}))) \right \vert\leq C\frac{\log(N)}{N}\Vert \psi \Vert_{C^5([-K,+K])}.$$
 \end{thm}
 {\it Proof}. In this proof, $C>0$ will denote a constant that does not depend on $m,N,\psi$ but may depend on $\mathscr{X}$. This generic constant $C$ will change from line to line. We start by using Proposition \ref{Approxtr} to approximate each operator $A_\gamma$ by a finite rank operator $A_\gamma^{(m)}$ with $\mathrm{rank}(A_\gamma^{(m)})\leq m$ in such a way that
 $$\max_{\gamma}\Vert A_\gamma -A^{(m)}_\gamma\Vert_1\leq C  e^{-am},$$
 for some $0<a\leq a_0$ which depends obviously on $\mathscr{X}$. We then set
 $$\mathscr{X}^{(m)}:=\sum_\gamma  A^{(m)}_\gamma.\gamma.$$
Note that since for any trace class operator $T$, we have $\Vert T^*\Vert_1=\Vert T\Vert_1$, 
 and using the fact that the adjoint of a finite rank operator is also finite rank with the same rank, we can also assume that each $\mathscr{X}^{(m)}$ is also self-adjoint.

 In practice, we will set $$m=[ \alpha \log(N)]$$ for a suitably chosen $\alpha>0$, see at the end of the proof.   Let $K>0$ be large enough so that for all $N$ and $m$ we have
 $$\Vert \pi_N(\mathscr{X}) \vert\leq K,\  \Vert \pi_N(\mathscr{X}^{(m)}) \vert\leq K,\ \mathrm{and}\ \Vert \lambda(\mathscr{X})\Vert\leq K,\  \Vert \lambda(\mathscr{X}^{(m)})\Vert\leq K.$$
 Such a $K$ does exist obviously by Lemma \ref{tensornorm1}.
 In the following we will also set 
 $$r:=\#\{ \gamma\ :\ A_\gamma \neq 0\}. $$
 Let us write
 $$\left \vert \frac{\Exp_N\left (\mathrm{Tr}(\pi_N(\mathscr{X})\psi(\pi_N(\mathscr{X})))\right)}{N}-\tau(\lambda(\mathscr{X})\psi(\lambda(\mathscr{X}))) \right \vert$$
 $$\leq \mathcal{E}_1(m,N)+ 
  \mathcal{E}_2(m,N)+ \mathcal{E}_3(m,N),$$
  where 
  $$ \mathcal{E}_1(m,N)=\frac{1}{N}\Exp_N\left (\left\vert \mathrm{Tr}(\pi_N(\mathscr{X})\psi(\pi_N(\mathscr{X}))))-\mathrm{Tr}(\pi_N(\mathscr{X}^{(m)})\psi(\pi_N(\mathscr{X}^{(m)}))) \right \vert\right),$$
  $$ \mathcal{E}_2(m,N)=\left \vert \frac{\Exp_N(\mathrm{Tr}(\pi_N(\mathscr{X}^{(m)})\psi(\pi_N(\mathscr{X}^{(m)}))))}{N}-\tau(\lambda( \mathscr{X}^{(m)}) \psi(\lambda( \mathscr{X}^{(m)})) ) \right \vert,$$
  $$  \mathcal{E}_3(m,N)=\left \vert  \tau(\lambda( \mathscr{X}^{(m)}) \psi(\lambda( \mathscr{X}^{(m)}))- \tau(\lambda(\mathscr{X})\psi(\lambda(\mathscr{X}))) \right \vert.$$
  We will now bound each error term $\mathcal{E}_j(m,N)$ for $j=1,2,3$ separately.
  
  \noindent {\bf Bounding $\mathcal{E}_1(m,N)$}.
  Writing
  $$S_{m,N}:=\frac{1}{N}\left\vert \mathrm{Tr}(\pi_N(\mathscr{X})\psi(\pi_N(\mathscr{X}))))-\mathrm{Tr}(\pi_N(\mathscr{X}^{(m)})\psi(\pi_N(\mathscr{X}^{(m)}))) \right \vert$$
  $$\leq \frac{1}{N}\left \vert \mathrm{Tr}\left( \pi_N(\mathscr{X}^{(m)}-\mathscr{X})\psi(\pi_N( \mathscr{X}^{(m)})) \right)   \right \vert $$
  $$+ \frac{1}{N}\left \vert \mathrm{Tr}\left( \pi_N(\mathscr{X})(\psi(\pi_N( \mathscr{X}^{(m)}))-\psi(\pi_N(\mathscr{X})))\right)\right\vert,$$
  which by applying the standard inequality $\vert \mathrm{Tr}(AB)\vert \leq \Vert A \Vert \Vert B\Vert_1$ for any operators $A,B$ with $A$ bounded and $B$ trace class gives
  $$S_{m,N}\leq  \frac{\Vert \pi_N(\mathscr{X}^{(m)})-\pi_N(\mathscr{X}) \Vert_1}{N}\Vert \psi(\pi_N( \mathscr{X}^{(m)}))\Vert$$
  $$+  \frac{\Vert \pi_N(\mathscr{X}) \Vert_1}{N}\Vert \psi(\pi_N( \mathscr{X}^{(m)}))-\psi(\pi_N(\mathscr{X}))\Vert.$$
  Applying Lemma \ref{tracenormbound} with Lemma \ref{Lipschitzbound} we end up with
 $$S_{m,N}\leq \Vert \psi \Vert_{C^0([-K,+K])}Cre^{-am}$$
 $$+M_0(K)\Vert \psi \Vert_{C^3([-K,+K])} \left (\sum_\gamma \Vert A_\gamma \Vert_1 \right)\Vert \pi_N( \mathscr{X}^{(m)})-\pi_N( \mathscr{X})\Vert.$$
 Note that to bound $\Vert \pi_N( \mathscr{X}^{(m)})-\pi_N( \mathscr{X})\Vert$, we write (using Lemma \ref{tensornorm1})
 $$\Vert \pi_N( \mathscr{X}^{(m)})-\pi_N( \mathscr{X})\Vert \leq \sum_\gamma \Vert (A_\gamma^{(m)}-A_\gamma)\otimes \pi_N(\gamma)\Vert
 \leq \sum_\gamma \Vert (A_\gamma^{(m)}-A_\gamma)\Vert\leq rCe^{-am}.$$
 We have obtained
 $$ \mathcal{E}_1(m,N)\leq C \Vert \psi \Vert_{C^3([-K,+K])}e^{-am}.$$
 
 \noindent {\bf Bounding $\mathcal{E}_3(m,N)$.} Applying Proposition \ref{tracialest2} combined with Lemma \ref{Lipschitzbound}, we have
 $$\mathcal{E}_3(m,N)=\left \vert  \tau(\lambda( \mathscr{X}^{(m)}) \psi(\lambda( \mathscr{X}^{(m)}))- \tau(\lambda(\mathscr{X})\psi(\lambda(\mathscr{X}))) \right \vert $$
 $$\leq C \Vert \psi \Vert_{C^3([-K,+K])} \Vert \lambda(\mathscr{X})-\lambda(\mathscr{X}^{(m)})\Vert +C\Vert \psi \Vert_{C^0([-K,+K])}\sum_\gamma \Vert A_\gamma-A_\gamma^{(m)}\Vert_1.$$
 We then apply Lemma \ref{tensornorm1} as above to bound $\Vert \lambda(\mathscr{X})-\lambda(\mathscr{X}^{(m)})\Vert$ and we finally get 
 $$ \mathcal{E}_3(m,N)\leq C \Vert \psi \Vert_{C^3([-K,+K])}e^{-am}.$$
 
 \noindent {\bf Bounding $\mathcal{E}_2(m,N)$.} This is where we will use the main input from the polynomial method as detailed in \cite{CGTVH},  which we rephrase here according to our notations.
 \begin{thm} \cite{CGTVH}, Corollary 7.2.
 Assume that $\mathscr{Y}=\sum_\gamma C_\gamma .\gamma \in \mathcal{M}_D(\C)[\mathbb{F}^d]$ is self-adjoint. Let $q_0=\max\{ \vert \gamma \vert\ :\ C_\gamma \neq 0\}$. Then there exists
a universal constant $\widetilde{C}>0$ such that we have for all $N>0$, for all $h\in C^\infty_0(\R)$,
$$\left\vert  \frac{\Exp_N(\mathrm{Tr}(h(\pi_N(\mathscr{Y}))))}{N}-\tau(h(\lambda(\mathscr{Y}))) \right \vert \leq \widetilde{C} 
\frac{D(q_0(1+\log(d)))^4}{N} \Vert f^{(5)}\Vert_{L^2([0,2\pi])},$$
where we have set for all $\theta\in \R$, $f(\theta)=h(K\cos(\theta))$.
\end{thm}
Note the important fact that all the dependencies in $\mathscr{Y}$ in the above estimate are given by $D$, $K$ and $q_0$. Our aim is to apply the above inequality
to $\mathscr{Y}_m$ which is related to $\mathscr{X}^{(m)}$ as follows. We write $\mathrm{supp}(\mathscr{X})=\{ \gamma\ :\ A_\gamma \neq 0\}$.
Set
$$\mathcal{F}_m:=\left (\bigcap_{\gamma \in \mathrm{supp}(\mathscr{X})} \mathrm{Ker}(A_\gamma^{(m)})\right)^\perp\subset \mathcal{H}.$$
Then by elementary Hilbert space theory we have (we use that $\mathscr{X}^{(m)}$ is self-adjoint)
$$\mathcal{F}_m=\bigoplus_{\gamma \in \mathrm{supp}(\mathscr{X})} \mathrm{Im}((A_\gamma^{(m)})^*)=\bigoplus_{\gamma \in \mathrm{supp}(\mathscr{X})} \mathrm{Im}(A_\gamma^{(m)}),$$
and is therefore a finite dimensional subspace with
$$\mathrm{dim}(\mathcal{F}_m)\leq \sum_{\gamma} \mathrm{rank}(A_\gamma^{(m)})\leq rm,$$
which is obviously invariant by all operators $A_\gamma^{(m)}$. We now choose a basis $\mathscr{B}_m$ of $\mathcal{F}_m$ and consider for each $\gamma$
$$M_\gamma^{(m)}:=\mathrm{Mat}_{\mathscr{B}_m}(A_\gamma^{(m)}\vert_{\mathcal{F}_m}),$$
then
$$\mathscr{Y}_m:=\sum_{\gamma} M_\gamma^{(m)} .\gamma \in \mathcal{M}_D(\C)[\mathbb{F}^d],$$
with $D\leq rm$. We then remark that if $A:\mathcal{H}\rightarrow \mathcal{H}$ is a finite rank operator satisfying
\begin{itemize}
\item $A(\mathcal{F}_M)\subset \mathcal{F}_m$
\item $\mathcal{F}_m^\perp\subset \mathrm{Ker}(A)$,
\end{itemize}
then we have $\mathrm{Tr}(A)=\mathrm{Tr}(A\vert_{\mathcal{F}_m})$.
Since for each $\gamma$, $A_\gamma^{(m)}(\mathcal{F}_m)\subset \mathcal{F}_m$ and $A_\gamma^{(m)}(\mathcal{F}_m^\perp)=\{0\}$, 
we have therefore
$$\mathrm{Tr}(A_\gamma^{(m)})= \mathrm{Tr}(A_\gamma^{(m)}\vert_{\mathcal{F}_m})=\mathrm{Tr}(M_\gamma^{(m)}).$$
This properties are preserved by composition and consequently we get for all $\gamma_1,\ldots,\gamma_p$,
$$\mathrm{Tr}(A_{\gamma_1}^{(m)}\ldots A_{\gamma_p}^{(m)})= \mathrm{Tr}(A_{\gamma_1}^{(m)}\ldots A_{\gamma_p}^{(m)}\vert_{\mathcal{F}_m})=
\mathrm{Tr}(M_{\gamma_1}^{(m)}\ldots M_{\gamma_p}^{(m)}).$$
It is now not difficult to see that for any polynomial which vanishes at $0$, say $P(X)=XQ(X)$ we have
$$\mathrm{Tr}(P(\pi_N(\mathscr{Y}_m)))=\mathrm{Tr}(P(\pi_N(\mathscr{X}^{(m)}))).$$
By density of polynomials and continuity of the functionals
$$Q\mapsto \tau (\pi_N(\mathscr{Y}_m)Q(\pi_N(\mathscr{Y}_m))),\  Q\mapsto \tau (\pi_N(\mathscr{X}_m)Q(\pi_N(\mathscr{X}_m))),$$
we can extend this property to any $h\in C_0^\infty(\R)$ which vanishes at $0$.  
Recalling that if $\mathscr{Z}=\sum_\gamma D_\gamma.\gamma \in S_1(\mathcal{H})[\mathbb{F}^d]$,
we have $$\tau(\lambda(\mathscr{Z}))=\mathrm{Tr}(D_{Id}),$$
we can prove similarly that for all polynomial $P$ which vanishes at $0$ we have
$$\tau(P(\lambda(\mathscr{Y}_m)))=\tau(P(\lambda(\mathscr{X}^{(m)}))),$$
and extend similarly the result to any $h\in C_0^\infty(\R)$ with $h(0)=0$.

We can now apply the above theorem to get (with $h=x\psi(x)$, $\mathscr{Y}=\mathscr{Y}_m$, $D=\mathrm{dim}(\mathcal{F}_m)$)
$$ \mathcal{E}_2(m,N)\leq C\frac{m}{N}\Vert f^{(5)}\Vert_{L^2([0,2\pi])}\leq C\alpha \frac{\log(N)}{N} \Vert f^{(5)}\Vert_{L^2([0,2\pi])}.$$
It is elementary to check that in our situation, by Fa\`a di Bruno's formula, 
$$\Vert f^{(5)}\Vert_{L^2([0,2\pi])}\leq C_K \Vert \psi \Vert_{C^5([-K,+K])},$$
where $C_K$ depends only on $K$. We now choose $\alpha$ large enough such that 
$$e^{-am}=e^{-a[\alpha \log(N)]}\leq \frac{C}{N},$$
and the proof is done. $\square$

We conclude this section by some remarks.
\begin{itemize}
\item If $\mathscr{X}=\sum_\gamma A_\gamma.\gamma \in S_1(\mathcal{H})[\mathbb{F}^d]$ but the $A_\gamma$'s are not in an exponential class but enjoy polynomial decay say
$$\mu_j(A_\gamma)\leq C (1+j)^{2+2\beta},$$ for some $\beta>0$, then we can still get an effective error term in Theorem \ref{main1} of size $O\left(N^{\frac{\beta}{1+\beta}}\right)$, by using similar ideas.
\item If $\mathscr{X}=\sum_\gamma A_\gamma.\gamma \in S_1(\mathcal{H})[\mathbb{F}^d]$ but no effective estimate on the singular value sequences of $A_\gamma$'s is available, 
we can still obtain the non effective convergence
$$\lim_{N\rightarrow +\infty} \frac{\Exp_N\left (\mathrm{Tr}(\pi_N(\mathscr{X})\psi(\pi_N(\mathscr{X})))\right)}{N}=\tau(\lambda(\mathscr{X})\psi(\lambda(\mathscr{X}))).$$
\end{itemize}

\section{An effective example}
As in the introduction, let $\Omega_0$ be the open disc given by 
$$\Omega_0=D\left(1,\frac{3}{2}\right)=\left \{z \in \C\ :\ \vert z-1\vert<\frac{3}{2}\right\}.$$
For all $j=1,\ldots,\infty$, set $\gamma_j(z)=\frac{1}{j+z}$. Then for all $j$, $\gamma_j$ satisfies
$$\gamma_j(\Omega_0)\subset D(1,1).$$
Let $m_0\leq j_1<j_2$ be integers. Here the free group on two generators is $$\mathbb{F}^2=\langle a_1,a_2;a_1^{-1},a_2^{-1}\rangle.$$
We will look at the twisted transfer operators
$$\lt_N=T_{\gamma_{j_1}}\otimes \pi_N(a_1)+T_{\gamma_{j_2}}\otimes \pi_N(a_2),$$
acting on $H^2(\Omega_0)\otimes \mathcal{V}_N$, and where $\pi_N(a_1)$, $\pi_N(a_2)$ are independent uniform random permutation matrices.
Here we have
$$\lt_N^*\lt_N=\left ( T_{\gamma_{j_1}}^*T_{\gamma_{j_1}}+T_{\gamma_{j_2}}^*T_{\gamma_{j_2}} \right)\otimes I_{\C^N}+ T_{\gamma_{j_1}}^*T_{\gamma_{j_2}}\otimes \pi_N(a_1^{-1}a_2)+
 T_{\gamma_{j_2}}^*T_{\gamma_{j_1}}\otimes \pi_N(a_2^{-1}a_1).$$
 The limit operator acting on $H^2(\Omega_0)\otimes \ell^2(\mathbb{F}^2)$ is
 $$\mathscr{M}_\infty=\left ( T_{\gamma_{j_1}}^*T_{\gamma_{j_1}}+T_{\gamma_{j_2}}^*T_{\gamma_{j_2}} \right)\otimes Id+ T_{\gamma_{j_1}}^*T_{\gamma_{j_2}}\otimes \lambda(a_1^{-1}a_2)+
 T_{\gamma_{j_2}}^*T_{\gamma_{j_1}}\otimes \lambda(a_2^{-1}a_1)=\lambda(\mathscr{X}),$$
 with $\mathscr{X}=\left ( T_{\gamma_{j_1}}^*T_{\gamma_{j_1}}+T_{\gamma_{j_2}}^*T_{\gamma_{j_2}} \right)Id+ T_{\gamma_{j_1}}^*T_{\gamma_{j_2}}.a_1^{-1}a_2+
 T_{\gamma_{j_2}}^*T_{\gamma_{j_1}}.a_2^{-1}a_1$.
 The goal of this section is to describe the measure $\mu_{\mathscr{X}}$ in the limit regime $m_0\rightarrow +\infty$ and deduce an effective Weyl law for the singular values of $\lt_N$.
 \subsection{Integral kernels and trace norm estimates}
 Using the Bergman kernel, we know that we have for all $z\in \Omega_0$, for all $f\in H^2(\Omega_0)$
 $$T_{\gamma_{i}}(f)(z)=\int_{\Omega_0}B_{\Omega_0}(\gamma_{i}(z),w)f(w)dm(w).$$
 A direct application of Fubini gives an integral representation for $T_{\gamma_i}^*(f)$ as
 $$T_{\gamma_{i}}^*(f)(z)=\int_{\Omega_0}B_{\Omega_0}(z,\gamma_{i}(w))f(w)dm(w). $$
 As a consequence we get for all $i,j$
 $$T_{\gamma_{i}}^*T_{\gamma_{j}}(f)(z)=\int_{\Omega_0}B_{\Omega_0}(z,\gamma_{i}(w))f(\gamma_j (w))dm(w).$$
 Recall that $\gamma_\ell(z)=\frac{1}{\ell+z}$, so uniformly on compact sets of $\Omega_0$ as $\ell\rightarrow \infty$ we have $\gamma_\ell(z)\rightarrow 0$. In particular we obtain that 
 as $i,j\rightarrow +\infty$, uniformly on compact sets, we have
 $$T_{\gamma_{i}}^*T_{\gamma_{j}}(f)(z)\rightarrow f(0)\int_{\Omega_0}B_{\Omega_0}(z,0)dm(w)=f(0)\mathrm{Vol}(\Omega_0)B_{\Omega_0}(z,0)=\frac{f(0)}{\left( 1+\frac{4}{9}(z-1)\right)^2}.$$
 Here we have used  the explicit formula
 $$B_{\Omega_0}(z,w)=\frac{4}{9\pi}\frac{1}{\left(1-\frac{4}{9}(z-1)\overline{(w-1)}\right)^2}.$$
 We will need the following fact.
 \begin{propo}
 \label{weaklimit1}
 Assume that $i,j\geq m_0$, then we have
 $$\lim_{m_0\rightarrow+\infty} \Vert  T_{\gamma_{i}}^*T_{\gamma_{j}}-\delta_0\Vert_1=0,$$
 where $\delta_0:H^2(\Omega_0)\rightarrow H^2(\Omega_0)$ is the rank one operator defined by
 $$\delta_0(f)=\frac{f(0)}{\left( 1+\frac{4}{9}(z-1)\right)^2}.$$
 
 \end{propo}
 \noindent {\it Proof}. As already explained in the proof of Proposition \ref{tracest1}, we have for any Hilbert basis $(e_\ell)_{\ell \in \N}$ of $H^2(\Omega_0)$, 
 $$\Vert  T_{\gamma_{i}}^*T_{\gamma_{j}}-\delta_0\Vert_1\leq \sum_{\ell=0}^\infty \Vert  T_{\gamma_{i}}^*T_{\gamma_{j}}(e_\ell)-\delta_0(e_\ell)\Vert_{H^2(\Omega_0)}.$$
 Here we will use the effective basis given by
 $$e_\ell(z)=\sqrt{\frac{\ell+1}{\pi}}\left ( \frac{z-1}{3/2}\right)^\ell.$$
For us the following fact will be enough.
\begin{lem}
There exists a universal constant $C>0$ such that for all $\ell$, for all $m_0\geq 1$ ,
$$\Vert  T_{\gamma_{i}}^*T_{\gamma_{j}}(e_\ell)-\delta_0(e_\ell)\Vert_{H^2(\Omega_0)}\leq C \sqrt{\frac{\ell+1}{\pi}}\ell \left(\frac{2}{3}\right)^\ell \frac{1}{m_0}\left (1+m_0^{-1}\right)^\ell.$$
\end{lem}
\noindent {\it Proof}. This is just a brute force estimate using the explicit formulas for $e_\ell(z)$ and $B_\Omega(z,w)$. We omit it. $\square$

As a consequence, for all $m_0>8$ we have the bound
$$\Vert  T_{\gamma_{i}}^*T_{\gamma_{j}}-\delta_0\Vert_1\leq C \frac{1}{m_0}\sum_{\ell=0}^\infty \ell \sqrt{\ell+1} \left (\frac{3}{4}\right)^\ell,$$
and clearly we have the desired result. $\square$

We point out that we have the identity
$$\delta_0(f)=\mathrm{Vol}(\Omega_0)B_{\Omega_0}(z,0)\langle f, B_{\Omega_0}(.,0)\rangle_{H^2(\Omega_0)},$$
from which we can observe that $\delta_0$ is self-adjoint and his spectrum is precisely
$$\mathrm{Sp}(\delta_0)=\{0\}\cup \left \{ \frac{81}{25}\right\}.$$
Indeed, this is a rank one operator and we have
$$\delta_0(B_{\Omega_0}(.,0))=B_{\Omega_0}(z,0)\mathrm{Vol}(\Omega_0)B_{\Omega_0}(0,0)=\left ( \frac{9}{5}\right)^2 B_{\Omega_0}(z,0).$$
As a simple consequence (by Lidskii's formula) we have for all $p\in \N$,
$$\mathrm{Tr}(\delta_0^p)=\left (\frac{81}{25}\right )^p.$$

\subsection{Limit spectrum of $\mathscr{M}_\infty$ as $m_0\rightarrow \infty$.}
By Lemma \ref{tensornorm1} and Proposition \ref{weaklimit1}  we now that in operator norm, we have the limit
$$\lim_{m_0\rightarrow+\infty} \Vert \mathscr{M}_\infty(m_0)-\lambda(\mathscr{Z}_0)\Vert=0,$$
where $\mathscr{Z}_0\in S_1(H^2(\Omega_0))[\mathbb{F}^2]$ is such that
$$\lambda(\mathscr{Z}_0)=\delta_0\otimes \left ( 2Id+U+U^*\right),$$
with $U=\lambda(a_1^{-1}a_2)$. Moreover, by combining Proposition \ref{tracialest2} with Proposition \ref{weaklimit1}, we now that for all $\varphi \in C_0^\infty(\R)$ with $\varphi(0)=0$
$$\lim_{m_0\rightarrow +\infty} \tau( \varphi( \mathscr{M}_\infty(m_0)))=\tau(\varphi(\lambda(\mathscr{Z}_0)))=\int \frac{\varphi(x)}{x}d\mu_{\mathscr{Z}_0}(x).$$
Note that this convergence is not uniform with respect to $\varphi$.
It turns out that $\mu_{\mathscr{Z}_0}$ is computable. Let $p\geq 1$ (recall we are only allowing test functions which vanish at $0$) we have by the binomial formula
$$\tau\left((\lambda(\mathscr{Z}_0))^p\right)=\mathrm{Tr}(\delta_0^p)\tau\left( (2Id+U+U^*)^p\right)$$
$$=\left (\frac{81}{25}\right )^p\sum_{k=0}^p \binom{p}{k}2^{p-k}\tau\left( (U+U^*)^k\right).$$
It is a classical fact that the moments of a "Haar unitary element" plus its adjoint can be explicitely computed, see for example \cite{NicaSpeicher} page 12-13. We have actually
$$\tau\left( (U+U^*)^k\right)=\left \{ 0\ \mathrm{if}\ k\ \mathrm{is\ odd,} \atop \binom{k}{k/2}\  \mathrm{if}\ k\ \mathrm{is\ even.} \right. $$
Moreover we have also the so-called "arcsine law"
$$\tau\left( (U+U^*)^k\right)=\frac{1}{\pi}\int_{-2}^{+2}t^k\frac{dt}{\sqrt{4-t^2}}.$$
We get therefore by using the binomial formula again
$$\tau\left((\lambda(\mathscr{Z}_0))^p\right)=\frac{1}{\pi}\int_{-2}^{+2}\left ( \frac{81}{25}(2+t)\right)^p\frac{dt}{\sqrt{4-t^2}},$$
which by an obvious affine change of variable gives
$$\tau\left((\lambda(\mathscr{Z}_0))^p\right)=\frac{25}{81\pi}\int_{0}^{\frac{324}{25}}u^p\frac{du}{\sqrt{4-(\frac{25}{81}u-2)^2}}.$$
By Stone-Weierstrass theorem we end up with the final formula
$$d\mu_{\mathscr{Z}_0}(x)= \frac{25}{81\pi}\chi_{[0,\frac{324}{25}]}(x)\frac{xdx}{\sqrt{4-(\frac{25}{81}x-2)^2}}.$$
Among other interesting facts, this formula shows that the spectrum of $\lambda(\mathscr{Z}_0)$ is the full interval 
$$\mathrm{Sp}(\lambda(\mathscr{Z}_0))=\left[0,\frac{324}{25}\right],$$ with
$\frac{324}{25}=12,96$. In particular the spectrum of $\mathscr{M}_\infty(m_0)$ converges {\it in Hausdorff distance} to this interval as $m_0\rightarrow+\infty$.
We have the following statement which is direct consequence of the above analysis and Theorem \ref{main1}.
\begin{cor}
Let $I=[a,b]$ with $0<a<b$ be any closed interval such that $I\subset (0,\frac{324}{25}]$. Then there exists $M_0,C_0,C_1>0$ such that for all $m_0\geq M_0$ fixed, for all $N$ large,
$$C_0 N\leq \Exp_N( \mathcal{N}_{I})\leq C_1 N,$$
 where as before, $\mathcal{N}_{I}:=\# \left\{ j\in \N\ :\ \mu_j^2(\lt_N)\in I \right\}$.
Both constants $C_0,C_1$ above can be taken arbitrarily close to 
$$\int_I  \frac{d\mu_{\mathscr{Z}_0}(x)}{x}= \frac{25}{81\pi} \int_a^b \frac{dx}{\sqrt{4-(\frac{25}{81}x-2)^2}}$$
$$=\frac{1}{\pi} \left[ \arcsin\left( \frac{25b - 162}{162} \right) - \arcsin\left( \frac{25a - 162}{162} \right) \right],$$
provided $m_0$ is large enough.
\end{cor}
\noindent{\it Proof}. Fix $\epsilon>0$. For all $\alpha>0$ small enough, there exist $\varphi^+_\alpha, \varphi^-_\alpha \in C_0^\infty((0,\infty))$ such that
$$0\leq \varphi_\alpha^-\leq \chi_I\leq \varphi^+_\alpha\leq 1, $$
with $\mathrm{Supp}(\varphi^+_\alpha)\subset [a-\alpha,b+\alpha]$, $\mathrm{Supp}(\varphi^-_\alpha)\subset [a,b]$ and $\varphi_\alpha^+\equiv 1$ on $I$ while
$\varphi_\alpha^-\equiv 1$ on $[a+\alpha,b-\alpha]$. Fix then $\alpha>0$ small enough so that we have
$$-\epsilon+\int_I  \frac{d\mu_{\mathscr{Z}_0}(x)}{x}\leq  \int  \frac{\varphi_\alpha^-(x)}{x}d\mu_{\mathscr{Z}_0}(x)\leq  \int \frac{\varphi_\alpha^+(x)}{x}d\mu_{\mathscr{Z}_0}(x)
\leq \int_I  \frac{d\mu_{\mathscr{Z}_0}(x)}{x}+\epsilon,$$
which is possible due to the fact that $\mu_{\mathscr{Z}_0}$ is absolutely continuous with respect to Lebesgue and hence non-atomic.
By weak convergence of $\mu_\infty(m_0)$ to $\mu_{\mathscr{Z}_0}$, we know that for all $m_0$ large enough we have
$$-\epsilon+\int  \frac{\varphi_\alpha^-(x)}{x}d\mu_{\mathscr{Z}_0}(x)\leq  \int  \frac{\varphi_\alpha^-(x)}{x}d\mu_{\infty}(x)\leq \int  \frac{\varphi_\alpha^+(x)}{x}d\mu_{\infty}(x)
\leq \int  \frac{\varphi_\alpha^+(x)}{x}d\mu_{\mathscr{Z}_0}(x)+\epsilon.$$
On the other hand we have for all $N$ large enough by Theorem \ref{main1}, 
$$-\epsilon+\int  \frac{\varphi_\alpha^-(x)}{x}d\mu_{\infty}(x)\leq \frac{\Exp_N(\mathrm{Tr}(\varphi^-_\alpha(\lt_N^*\lt_N)))}{N}$$
$$ \leq \frac{\Exp_N(\mathcal{N}_I)}{N}\leq \frac{\Exp_N(\mathrm{Tr}(\varphi^+_\alpha(\lt_N^*\lt_N)))}{N}\leq \int  \frac{\varphi_\alpha^+(x)}{x}d\mu_{\infty}(x)+\epsilon.$$
The above inequalities show that for all $N$ large we have indeed
$$-3\epsilon+\int_I  \frac{d\mu_{\mathscr{Z}_0}(x)}{x}\leq \frac{\Exp_N(\mathcal{N}_I)}{N}\leq \int_I  \frac{d\mu_{\mathscr{Z}_0}(x)}{x}+3\epsilon,$$
the claim is proved. $\square$


\begin{thebibliography}{10}

\bibitem{AFW}
Jean~Francois Arnoldi, Fr\'ed\'eric Faure, and Tobias Weich.
\newblock Asymptotic spectral gap and {W}eyl law for {R}uelle resonances of
  open partially expanding maps.
\newblock {\em Ergodic Theory Dynam. Systems}, 37(1):1--58, 2017.

\bibitem{Baladi1}
Viviane Baladi.
\newblock {\em Positive transfer operators and decay of correlations},
  volume~16 of {\em Advanced Series in Nonlinear Dynamics}.
\newblock World Scientific Publishing Co., Inc., River Edge, NJ, 2000.

\bibitem{Bandtlow1}
Oscar~F. Bandtlow.
\newblock Resolvent estimates for operators belonging to exponential classes.
\newblock {\em Integral Equations Operator Theory}, 61(1):21--43, 2008.

\bibitem{BJ1}
Oscar~F. Bandtlow and Oliver Jenkinson.
\newblock Explicit eigenvalue estimates for transfer operators acting on spaces
  of holomorphic functions.
\newblock {\em Adv. Math.}, 218(3):902--925, 2008.

\bibitem{BJ2}
Oscar~F. Bandtlow and Oliver Jenkinson.
\newblock On the {R}uelle eigenvalue sequence.
\newblock {\em Ergodic Theory Dynam. Systems}, 28(6):1701--1711, 2008.

\bibitem{BN}
Oscar~F. Bandtlow and Fr{\'e}d{\'e}ric Naud.
\newblock Lower bounds for the ruelle spectrum of analytic circle maps.
\newblock {\em Ergodic Theory Dynam. Systems}, 39 (2):289--310, 2019.

\bibitem{BCZ1}
Anirban Basak, Nicholas Cook, and Ofer Zeitouni.
\newblock Circular law for the sum of random permutation matrices.
\newblock {\em Electron. J. Probab.}, 23:Paper No. 33, 51, 2018.

\bibitem{BD1}
Anirban Basak and Amir Dembo.
\newblock Limiting spectral distribution of sums of unitary and orthogonal
  matrices.
\newblock {\em Electron. Commun. Probab.}, 18:no. 69, 19, 2013.

\bibitem{Bill1}
Patrick Billingsley.
\newblock {\em Probability and measure}.
\newblock Wiley Series in Probability and Mathematical Statistics. John Wiley
  \& Sons, Inc., New York, third edition, 1995.
\newblock A Wiley-Interscience Publication.

\bibitem{BC1}
Charles Bordenave and Beno\^it Collins.
\newblock Eigenvalues of random lifts and polynomials of random permutation
  matrices.
\newblock {\em Ann. of Math. (2)}, 190(3):811--875, 2019.

\bibitem{Bowen_Book}
Rufus Bowen.
\newblock {\em Equilibrium states and the ergodic theory of {A}nosov
  diffeomorphisms}, volume 470 of {\em Lecture Notes in Mathematics}.
\newblock Springer-Verlag, Berlin, revised edition, 2008.
\newblock With a preface by David Ruelle.

\bibitem{MCN1}
Irving Calderon, Michael Magee, and Fr\'ed\'eric Naud.
\newblock Spectral gap for random schottky surfaces.
\newblock {\em To appear in Analysis and PDE}, 2024.

\bibitem{CGTVH}
Chi-Fang Chen, Jorge Garza-Vargas, Joel~A. Tropp, and Ramon van Handel.
\newblock A new approach to strong convergence.
\newblock {\em Ann. of Math. (2)}, 203(2), 2026.

\bibitem{Collins1}
Beno\^it Collins and Camille Male.
\newblock The strong asymptotic freeness of {H}aar and deterministic matrices.
\newblock {\em Ann. Sci. \'Ec. Norm. Sup\'er. (4)}, 47(1):147--163, 2014.

\bibitem{Davies}
E.~B. Davies.
\newblock {\em Spectral theory and differential operators}, volume~42 of {\em
  Cambridge Studies in Advanced Mathematics}.
\newblock Cambridge University Press, Cambridge, 1995.

\bibitem{DJPP}
Ioana Dumitriu, Tobias Johnson, Soumik Pal, and Elliot Paquette.
\newblock Functional limit theorems for random regular graphs.
\newblock {\em Probab. Theory Related Fields}, 156(3-4):921--975, 2013.

\bibitem{GKZ1}
Alice Guionnet, Manjunath Krishnapur, and Ofer Zeitouni.
\newblock The single ring theorem.
\newblock {\em Ann. of Math. (2)}, 174(2):1189--1217, 2011.

\bibitem{HS1}
B.~Helffer and J.~Sj\"ostrand.
\newblock \'equation de {S}chr\"odinger avec champ magn\'etique et \'equation
  de {H}arper.
\newblock In {\em Schr\"odinger operators ({S}\o nderborg, 1988)}, volume 345
  of {\em Lecture Notes in Phys.}, pages 118--197. Springer, Berlin, 1989.

\bibitem{Krantz_Bergman}
Steven~G. Krantz.
\newblock {\em Geometric analysis of the {B}ergman kernel and metric}, volume
  268 of {\em Graduate Texts in Mathematics}.
\newblock Springer, New York, 2013.

\bibitem{MN1}
Michael Magee and Fr\'ed\'eric Naud.
\newblock Explicit spectral gaps for random covers of {R}iemann surfaces.
\newblock {\em Publ. Math. Inst. Hautes \'Etudes Sci.}, 132:137--179, 2020.

\bibitem{Mayer}
Dieter~H. Mayer.
\newblock On the thermodynamic formalism for the {G}auss map.
\newblock {\em Comm. Math. Phys.}, 130(2):311--333, 1990.

\bibitem{Moaz}
Yotam Moaz.
\newblock Asymptotic independence for random permutations from surface groups.
\newblock {\em To appear in Geometriae Dedicata}, 2023.

\bibitem{Naud1}
Fr\'ed\'eric Naud.
\newblock Random {C}overs of {C}ompact {S}urfaces and {S}mooth {L}inear
  {S}pectral {S}tatistics.
\newblock {\em Ann. Henri Poincar\'e}, 27(1):347--373, 2026.

\bibitem{Nica2}
Alexandru Nica.
\newblock Asymptotically free families of random unitaries in symmetric groups.
\newblock {\em Pacific J. Math.}, 157(2):295--310, 1993.

\bibitem{Nica1}
Alexandru Nica.
\newblock On the number of cycles of given length of a free word in several
  random permutations.
\newblock {\em Random Structures Algorithms}, 5(5):703--730, 1994.

\bibitem{NicaSpeicher}
Alexandru Nica and Roland Speicher.
\newblock {\em Lectures on the combinatorics of free probability}, volume 335
  of {\em London Mathematical Society Lecture Note Series}.
\newblock Cambridge University Press, Cambridge, 2006.

\bibitem{PP}
William Parry and Mark Pollicott.
\newblock Zeta functions and the periodic orbit structure of hyperbolic
  dynamics.
\newblock {\em Ast\'erisque}, (187-188):268, 1990.

\bibitem{Pedersen}
Gert~K. Pedersen.
\newblock Operator differentiable functions.
\newblock {\em Publ. Res. Inst. Math. Sci.}, 36(1):139--157, 2000.

\bibitem{PS1}
Anke Pohl and Louis Soares.
\newblock Density of resonances for covers of {S}chottky surfaces.
\newblock {\em J. Spectr. Theory}, 10(3):1053--1101, 2020.

\bibitem{PM1}
Ch. Pommerenke.
\newblock {\em Boundary behaviour of conformal maps}, volume 299 of {\em
  Grundlehren der mathematischen Wissenschaften [Fundamental Principles of
  Mathematical Sciences]}.
\newblock Springer-Verlag, Berlin, 1992.

\bibitem{PZ1}
Doron Puder and Tomer Zimhoni.
\newblock Local statistics of random permutations from free products.
\newblock {\em Int. Math. Res. Not. IMRN}, (5):4242--4300, 2024.

\bibitem{Ruelle}
David Ruelle.
\newblock Zeta-functions for expanding maps and {A}nosov flows.
\newblock {\em Invent. Math.}, 34(3):231--242, 1976.

\bibitem{BSimon}
Barry Simon.
\newblock {\em Trace ideals and their applications}, volume 120 of {\em
  Mathematical Surveys and Monographs}.
\newblock American Mathematical Society, Providence, RI, second edition, 2005.

\end{thebibliography}
\end{document}